\documentclass[12pt]{amsart}

\input xy
\usepackage[english]{babel}
\usepackage[latin1]{inputenc}
\usepackage{graphicx}
\usepackage{amsmath}
\usepackage{amssymb}
\usepackage{amscd}
\usepackage{color}
\usepackage{oldgerm}
\usepackage{amsfonts}
\usepackage{newlfont}
\usepackage{longtable}
\usepackage{multirow}

\usepackage[all]{xy}

\usepackage{upref}

\def\C{\Bbb C}

\def\T{\Bbb T}
\DeclareOldFontCommand{\rm}{\normalfont\rmfamily}{\mathrm}

\def\Aut{\operatorname{Aut}}

\def\det{\operatorname{det}}
\def\Diag{\operatorname{diag}}
\def\dim{\operatorname{dim}}

\def\End{\operatorname{End}}
\def\GL{\operatorname{GL}}
\def\Tr{\operatorname{Tr}}

\def\Ker{\operatorname{Ker}}

\def\Mat{\operatorname{Mat}}
\def\Rank{\operatorname{rank}}

\def\adj{\operatorname{adj}}
\def\Mat{\operatorname{Mat}}

\def\g{\frak g}

\def\h{\frak h}

\def\c{\frak{c}}

\theoremstyle{plain}\swapnumbers

{\newtheorem{Theorem}{Theorem}[section]
\newtheorem{Lemma}[Theorem]{Lemma}
\newtheorem{Prop}[Theorem]{Proposition}

\newtheorem{Cor}[Theorem]{Corollary}

\newtheorem{Remark}{Remark}

\setcounter{page}{1}

\title{On $3$-dimensional complex Hom-Lie algebras}

\author[R. Garc\'ia-Delgado, G. Salgado, O.A. S\'anchez-Valenzuela]{R. Garc\'ia-Delgado$^{(2)}$, G. Salgado$^{(1)}$
and O.A. S\'anchez-Valenzuela$^{(2)}$}

\address{(1) Fac. de Ciencias, UASLP, Av. Salvador Nava s/n, Zona Universitaria,
CP 78290, San Luis Potos\'{\i}, S.L.P., M\'exico.}
\address{(2) Centro de Investigaci\'on en Matem\'aticas, A.C.., Unidad M\'erida;
Yuc, M\'exico}

\email{rosendo.garcia@cimat.mx,rosendo.garciadelgado@alumnos.uaslp.edu.mx}
\email{gsalgado@fciencias.uaslp.mx, gil.salgado@gmail.com}
\email{adolfo@cimat.mx}



\date{\today}

\begin{document}

\maketitle

\begin{abstract}
We study and classify the 
$3$-dimensional Hom-Lie algebras over $\mathbb{C}$. 
We provide first a complete set of representatives for the isomorphism classes
of skew-symmetric bilinear products defined on a $3$-dimensional complex vector space $\g$.
The well known Lie brackets for the $3$-dimensional Lie algebras
are included into appropriate isomorphism classes of such products representatives.
For each product representative, we provide a complete set of canonical forms
for the linear maps $\g\to\g$ that turn $\g$ into a Hom-Lie algebra,
thus characterizing the corresponding isomorphism classes.
As by-products, Hom-Lie algebras for which the linear maps $\g\to\g$ are not 
homomorphisms for their products, are exhibited. 
Examples also arise of non-isomorphic families of Hom-Lie algebras
which share, however, a fixed Lie-algebra product on $\g$.
In particular, this is the case for the complex simple Lie algebra $\frak{sl}_2(\C)$. 
Similarly, there are isomorphism classes for which their skew-symmetric 
bilinear products can never be Lie algebra brackets on $\g$.
\end{abstract}

\section*{Introduction}

The main difference between a Lie algebra and a Hom-Lie algebra is the
appearance of a linear endomorphism in the algebraic condition
that plays the role of the Jacobi identity, thus generalizing it. 
This map is called {\it the twist map\/.}
If this map is the identity, the well known Jacobi identity holds true,
and the product in the Hom-Lie algebra is a  Lie bracket. 
The problem of classifying Lie algebras is far away to have a solution,
and the same is true for Hom-Lie  algebras.
Nevertheless, there are tools to handle and accomplish a classification
under certain circumstances  that restrict but focus the problem.
For instance, the restrictions may respond to, or may be motivated by, 
geometric reasons. 
The purpose of this work is precisely to restrict the classification problem of Hom-Lie algebras 
to $3$-dimensional settings that might be succeptible of geometrical interpretations. 

\smallskip
Most of the literature in the subject deals with
special bilinear, skew-symmetric products ---namely, 
Lie brackets or deformations of them--- and the twist maps used
are homomorphisms for these special products.
For results under these hypotheses we refer the reader to
\cite{Alejandra}, \cite{Remm} or \cite{Xiao}. 
An interesting classification
based on more general grounds
has been given in \cite{Xie}, where even though the products
are Lie brackets for semisimple Lie algebras,
the authors do not
require the twist maps to be product homomorphisms.
So far, however, not too much has been said  
for the most general products and twist maps.
This work fills in this gap, at least when $\dim\g=3$, by dealing with 
products which are not restricted to be Lie brackets and twist maps that are not
necessarily product homomorphisms.

\smallskip
{\bf Convention.}
We shall consider the classification problem for
{\bf complex} $3$-dimensional vector spaces.
However, one can apply essentially the same arguments
under the slightly more general hypothesis of 
using any algebraically closed ground field
of characteristic zero.
We have chosen to work with the complex numbers as
more familiar choices can be made for some of the entries of the
canonical forms of the products or the twist maps
({\it ie\/,} one may use $i={\sqrt{-1}}$).

\smallskip
A {\it Hom-Lie algebra\/,} or {\it HL-algebra} 
for short, is a triple $(\g, \mu, T)$, where $\g$ is a vector space, 
$\mu:\g \times \g \to \g$ is a skew-symmetric bilinear map, 
and $T:\g\to\g$ is a linear endomorphism
---usually called {\it the twist map\/}--- satisfying the {\it HL-Jacobi identity\/:}
\begin{equation}\label{HL-JI}
\mu\left(T(x), \mu(y,z)\right) + \mu\left(T(y), \mu(z,x)\right) + \mu\left(T(z), \mu(x,y)\right)=0,
\end{equation}
for all $x,y,z\in \g$.  
Let $(\g, \mu_{\g},T)$ and $(\h, \mu_{\h},S)$ be HL-algebras. 
An {\it HL-morphism\/} between them is a linear map
$\varphi: \g \to \h$ satisfying:
\begin{itemize}
\item[(i)] $\varphi\left(\mu_{\g}(x,y)\right) = 
\mu_{\h}\left(\varphi(x), \varphi(y)\right)\,,$ for all $x,y \in \g$, and
\item[(ii)]  $\varphi \circ T = S \circ \varphi$.
\end{itemize}

In order to classify HL-algebras we look first at the 
$\GL(\g)$-orbits in the space of skew-symmetric bilinear products
$\mu:\g\times\g\to\g$ under the {\it left\/} $\GL(\g)$-action
$\mu\mapsto g\!\cdot\!\mu$, given by,
\begin{equation}\label{GL-accion}
(g\!\cdot\!\mu)(x,y)= g\left(\mu\left(g^{-1}(x),g^{-1}(y)\right)\right),\quad g\in\GL(\g),
\ \text{and\ }\ x,y\in\g.
\end{equation}
Skew-symmetric bilinear products in a $3$-dimensional vector space
$\g$ were partially classified in \cite{H-M-MM} for this action.
Actually, the authors classified the so called {\it non-degenerate\/} $\mu$'s 
(see {\bf Prop. \ref{multiplicaciones-casos-NO-degenerados}} below).
We have worked out the classification of  the {\it degenerate\/} $\mu$'s
(see {\bf Prop. \ref{multiplicaciones-casos degenerados}}),
thus ending up with  a complete set of representatives for the bilinear,
skew-symmetric products $\mu$ under the given $\GL(\g)$-action. 
Then, for each fixed $\mu$, we provide complete lists of 
canonical forms for the linear maps $T:\g\to\g$ in the vector subspace,
$$
\operatorname{HL}(\mu)=\left\{ T\in\End\g\mid 
\sum_{\text{cyclic}}\mu\left(T(x),\mu(y,z)\right) =0\,\right\}\,,
$$
under the left action $T\mapsto g \cdot T=g\,\circ\,T\,\circ\,g^{-1}$ of the
corresponding isotropy subgroups 
$G_\mu=\{\g\in\GL(\g)\mid g\!\cdot\!\mu=\mu\,\}$.
(See {\bf Propositions \ref{ND-1}}, {\bf\ref{ND-2}}, {\bf\ref{ND-3}}, {\bf\ref{D2-1}}, {\bf\ref{D2-2}}, 
{\bf\ref{D2-3}}, {\bf\ref{D2-4}}, {\bf\ref{D1-1}}, {\bf\ref{D1-2}} and {\bf \ref{D1-3}}, below).

\smallskip
A word has to be said about the non-triviality of the vector subspace 
$\operatorname{HL}(\mu)$ for a 
given skew-symmetric bilinear product $\mu:\g\times\g\to\g$.
That is, one would like to make sure that there are non-zero linear maps in 
$\operatorname{HL}(\mu)$;
at least for the case addressed in this work of
a complex $3$-dimensional
space $\g$. 
This has been proved in \cite{Remm} 
for {\it any\/} skew-symmetric product $\mu$
(see {\bf Thm.} 17 in \cite{Remm}).
Nevertheless, here 
is an alternative proof:
Since $\mu$ is skew-symmetric, the 
$4$-linear map
$(x, y, z, T) \mapsto 
\mu(T(x),\mu(y,z))+\mu(T(y),\mu(z,x))+\mu(T(z),\mu(x,y))$,
is alternating in the arguments $(x,y,z)$. Therefore,
there exists a bilinear map $\hat{\mu}:\wedge^3\g\times\End(\g)\to\g$, such that,
\begin{equation}\label{muencuatrovaribales}
\hat{\mu}(x \wedge y \wedge z, T)=\sum_{\circlearrowleft}\mu(T(x),\mu(y,z)).
\end{equation}
For a fixed triple $(x,y,z)\in\g\times\g\times\g$, this yields a linear
map $\End(\g)\ni T\mapsto \sum_{\circlearrowleft}\mu(T(x),\mu(y,z))\in\g$.
By changing the triple $(x,y,z)$ to $(x^\prime,y^\prime,z^\prime)$
through a linear map $g:\g\to \g$, the right hand side of
\eqref{muencuatrovaribales}, now written for $(x^\prime,y^\prime,z^\prime)$,
only introduces the scalar factor $\det g$.
Letting $\{x,y,z\}=\{e_1,e_2,e_3\}$ be 
a basis of $\g$, one obtains a linear map 
$\End(\g)\ni T\mapsto \hat{\mu}(e_1\wedge e_2\wedge e_3, T)\in\g$
whose kernel is precisely $\operatorname{HL}(\mu)$. 
Actually, {\it for a given\/} $\mu$, the map $T$ defined by
$T(e_1)=\mu(e_2,e_3)$, $T(e_2)=\mu(e_3,e_1)$
and $T(e_3)=\mu(e_1,e_2)$, lies in $\operatorname{HL}(\mu)$.

\section{Skew-symmetric bilinear products on a $3$-dimensional complex vector space}

Let $\g$ be a $3$-dimensional complex vector space and let
$\mu:\g \times \g \rightarrow \g$ be a  {\bf skew-symmetric bilinear map}. 
We shall call such a $\mu$ {\it a product in\/} $\g$.
Given a basis $\{e_1,e_2,e_3\}$ of $\g$, one obtains a one-to-one
correspondence $\mu\leftrightarrow M_\mu=(\mu_{ij})$
between products in $\g$ and $3\times 3$ complex matrices,
as follows:
$$
\mu(e_2,e_3)  = \sum_{i=1}^3\mu_{i1}\,e_i\,,\quad
\mu(e_3,e_1)  = \sum_{i=1}^3\mu_{i2}\,e_i\,,\quad
\mu(e_1,e_2)  = \sum_{i=1}^3\mu_{i3}\,e_i\,.
$$
It is straightforward to see that the $\GL(\g)$-action \eqref{GL-accion}
gets transformed into the corresponding
$\GL_3(\C)$-action $M_{\mu}\mapsto g\!\cdot\!M_{\mu}=M_{g\cdot\mu}$
given by,
\begin{equation}\label{matrix-action}
M_{g\cdot\mu}=(\det g)^{-1}\,g\,M_\mu\,g^t\,,
\quad
\text{where,}\quad g=(g_{ij})\in\GL_3(\C),
\end{equation}
and the entries $g_{ij}$ are taken from $g(e_j)=\sum_{i=1}^3 g_{ij}e_i$
for each $g\in\GL(\g)$ as usual.
At this point we may 
simplify the notation and write $\mu$
for the matrix $M_\mu$ itself. 
The classification problem is that of finding canonical forms
of $3\times 3$ complex matrices $\mu$ under the 
$\GL(\g)\simeq\GL_3(\Bbb C)$ action,
\begin{equation}\label{accionmatrizproducto}
g\!\cdot\!\mu = (\det g)^{-1}\,g\,\mu\,g^t .
\end{equation}
It is clear that this action preserves the symmetric and the skew-symmetric
components of $\mu$. Thus, we may decompose $\mu$ in the form,
\begin{equation}\label{decomp-sim-antisim}
\mu = S_\mu + A_\mu, \qquad
\text{with\ \ }(S_\mu)^t=S_\mu\ \ \text{and}\ \ (A_\mu)^t=-A_\mu\,.
\end{equation}
Since $\dim\g=3$, we have a one-to-one correspondence
$A_\mu\leftrightarrow \text{\bf a}_\mu\in\Bbb C^3$:
\begin{equation}\label{corresp-antisim-vector}
A_\mu = 
\begin{pmatrix}
\,\,0 & \!-a_3 & \,\,a_2 \\
\,\,a_3 & \,\,0 & \!-a_1\\
\!-a_2 & \,\,a_1 & \,\,0
\end{pmatrix}
\ \leftrightarrow\ \text{\bf a}_\mu=\begin{pmatrix} a_1\\ a_2\\ a_3\end{pmatrix}\in\Bbb C^3,
\end{equation}
with,
\begin{equation}\label{accion-sim-antisim}
g\!\cdot\!\mu  = 
(\det g)^{-1}\,\left(\, g\,S_\mu\,g^t + g\,A_\mu\,g^t
\,\right)
=S_{g\cdot\mu}+A_{g\cdot\mu}.
\end{equation}

\smallskip
{\bf Notation:} Write $g_{\,i\,\ast}\in\Bbb C^3$ for the vector obtained from the $i$th row
of the matrix $g\in\GL_3(\Bbb C)$ (equivalently, the $i$-th column of $g^t$):
\begin{equation}\label{vectores-de-g}
g_{\,1\,\ast}=\begin{pmatrix}
g_{11}\\ g_{12}\\ g_{13}
\end{pmatrix},
\quad
g_{\,2\,\ast}=\begin{pmatrix}
g_{21}\\ g_{22}\\ g_{23}
\end{pmatrix},
\quad
g_{\,3\,\ast}=\begin{pmatrix}
g_{31}\\ g_{32}\\ g_{33}
\end{pmatrix}.
\end{equation}
It is a straightforward computation to see that the action of $g$
in the skew-symmetric component of $\mu$ yields the correspondence,
\begin{equation}\label{accion-vector-antisim2}
(\det g)^{-1}g\,A_\mu\,g^t =A_{g\cdot \mu}\,
\leftrightarrow\,\text{\bf a}_{\,g\cdot\mu}=(\det g)^{-1}
\!
\begin{pmatrix}
\langle\text{\bf a}_\mu\times g_{\,2\,\ast}\,,\,g_{\,3\,\ast}\rangle \\
\langle\text{\bf a}_\mu\times g_{\,3\,\ast}\,,\,g_{\,1\,\ast}\rangle \\
\langle\text{\bf a}_\mu\times g_{\,1\,\ast}\,,\,g_{\,2\,\ast}\rangle
\end{pmatrix},
\end{equation}
where, $\langle\,\cdot\,,\,\cdot\,\rangle:\g\times\g\to\Bbb C$ denotes the
usual scalar product on the $3$-dimensional space $\g$.
Notice in particular that,
$$
\det g  = \langle \,g_{\,1\,\ast}\times g_{\,2\,\ast}\,,\,g_{\,3\,\ast}\, \rangle
 = \langle \,g_{\,2\,\ast}\times g_{\,3\,\ast}\,,\,g_{\,1\,\ast}\, \rangle
 = \langle \,g_{\,3\,\ast}\times g_{\,1\,\ast}\,,\,g_{\,2\,\ast}\, \rangle\,.
$$
\begin{Cor}
{\sl Assume $A_\mu\ne 0$. The choice $g_{\,i\,\ast}= \text{\bf a}_\mu$, with appropriate completion
for the matrix $g\in \GL_3(\Bbb C)$, 
brings $\text{\bf a}_{\,g\cdot\mu}$ into a vector
whose $i$-th component is equal to $1$ and the others are equal to zero\/.}
\end{Cor}
In order to produce appropriate canonical forms for $\mu=S_\mu+A_\mu$
under the $\GL_3(\Bbb C)$-action \eqref{accionmatrizproducto}, we shall proceed as follows:
First, assume that $g$ is chosen so as to bring
$S_{g\cdot \mu}$ into some {\it canonical form\/}; say $S$.
Then, restrict the action to the isotropy subgroup,
$$
G_S=\{g\in\GL_3(\Bbb C)\mid S=(\det g)^{-1}\,g\,S\,g^t\,\},
$$
and look at the $G_S$-orbits,
$\{g \cdot \mu=S+(\det g)^{-1}\,g\,A_\mu\,g^t,\ g\in G_S\}$.
Following \cite{H-M-MM} the product $\mu$ is called {\it non-degenerate\/} if $S_\mu$ 
is non-degenerate. Otherwise, $\mu$ is called {\it degenerate\/.}

\begin{Remark}{\rm
There is a nice characterization for $3$-dimensional Lie algebras.
Fix a basis $\{e_1,e_2,e_3\}$ of $\g$ and let $\{e_1^*,e_2^*,e_3^*\}$ be
its dual basis.
Fix the isomorphism $\wedge^2\g^*\to\g$ given by
$e_2^*\wedge e_3^*\mapsto e_1$,
$e_3^*\wedge e_1^*\mapsto e_2$ and $e_1^*\wedge e_2^*\mapsto e_3$.
Any element in $\wedge^2\g^*$ is decomposable. 
Then, for each $x\in\g$, there are a couple of
dual vectors $u_x$ and $v_x$ in $\g^*$
such that $u_x\wedge v_x\mapsto x$.
The ambiguity in the choice of $u_x$ and $v_x$ in $\g^*$ is
similar to that of the cross-product in $\Bbb C^3$:
the $2$-dimensional subspace generated by $u_x$ and $v_x$ can be rotated
and the vectors can be changed by scale transformations
$u_x\mapsto \lambda u_x$ and $v_x\mapsto \lambda^{-1}v_x$\,.
Now define a bilinear map
$B_\mu:\g\times\g\to\Bbb C$ as follows:
$$
\aligned
B_\mu(x,y) & 
=\mu^*(u_x\wedge v_x,u_y\wedge v_y)
\\
& :=\mu^*(u_x,u_y)\,\mu^*(v_x,v_y)-
\mu^*(u_x,v_y)\,\mu^*(v_x,u_y)\,,
\endaligned
$$
where,  $\mu^*:\g^*\times\g^*\to\C$ is the 
bilinear form in $\g^*$ defined by, 
$$
\mu^*(u,v)=\sum_{\circlearrowleft} u(\mu(e_2,e_3))\,v(e_1),
\quad\forall\,u,v\in\g^*.
$$
Observe in particular that 
$\mu^{*}(e_i^{*},e_j^{*})=\mu_{ij}$, for all $1 \leq i,j \leq 3$ and that
the matrix of $B_{\mu}$ in the basis $\{e_1,e_2,e_3\}$, is
the matrix of cofactors of $\mu$.
Now, a straightforward computation yields,
$$
\sum_{\circlearrowleft}\mu\left(e_1,\mu(e_2,e_3)\right)
= 
- \sum_{\circlearrowleft}\left(B_\mu(e_2,e_3)-B_\mu(e_3,e_2)\right)e_1\,.
$$
It is now immediate to conclude the following:

\begin{Prop} \label{ProductosDeLie}
{\sl Let $\mu$ be a skew-symmetric bilinear map on a $3$-dimensional
vector space $\g$. Then $(\g, \mu)$ is a Lie algebra if, and only if
its associated bilinear map $B_\mu$ is symmetric. In particular, if $\det \mu \neq 0$, then $(\g,\mu)$ is a Lie algebra if and only if the matrix of $\mu$ is symmetric\/.}
\end{Prop}

It also follows that:

\begin{Cor}\label{Gil}
{\sl Let $(\g,\mu)$ be a $3$-dimensional Lie algebra and $T\in \End_{\C}(\g)$. 
Then, the pair $(\mu,T)$ yields an HL
algebra structure in $\g$ if and only if, for all $x,y\in \g$,
$B_\mu(T(x),y) = B_\mu (x, T(y))$.\/}
\end{Cor}
}
\end{Remark}

This hints into the geometric role that the twist maps 
might play when $\mu$ is a Lie bracket. 
The best example at hand is $\g=\frak{sl}_2(\C)$. In this case
$B_\mu$ is a scalar multiple of the Cartan-Killing 
form and the twist maps must be self adjoint operators for
this invariant quadratic form (see {\bf Remark 5} after {\bf Prop. \ref{ND-2}} below).

\section{Non-degenerate products}

Write $\mu=S_\mu+A_\mu$ as before and suppose that $S_\mu$ is {\it non-degenerate\/}.
Fix first the canonical form $S=1\!\!1_3$ ({\it ie\/,} identity $3\times 3$ matrix)
so that, $G_{1\!\!1_3}=\operatorname{SO}_3(\C)$. Then,
$g\cdot \mu=1\!\!1_3+g A_{\mu} g^{-1}$,
for all $g \in G_{1\!\!1_3}$. In particular, $A_\mu$ and $A_{g \cdot \mu}$
have the same characteristic polynomial:
$\det(x\,1\!\!1_3-A_\mu)=x^3+\sigma(\mu)\,x$.
The coefficient $\sigma(\mu)$ is then an invariant in the $G_{1\!\!1_3}$-orbit.
It is easy to see that, 
$\sigma(\mu)= a_1^2+a_2^2+a_3^2=
\langle\,\text{\bf a}_\mu\,,\,\text{\bf a}_\mu\,\rangle$,
with $A_\mu \leftrightarrow\text{\bf a}_\mu 
= \left(\begin{smallmatrix} a_1\\ a_2\\ a_3 \end{smallmatrix} \right)$.
If $\sigma(\mu)\ne 0$, 
we may choose 
$\{g_{\,1\,\ast},g_{\,2\,\ast},g_{\,3\,\ast}\}$ to be an orthonormal basis with,
$g_{\,1\,\ast}={\sigma(\mu)}^{-1/2}\,\text{\bf a}_\mu$, and
$A_{g\cdot\mu}\leftrightarrow
\text{\bf a}_{g\cdot\mu}=
\left(\begin{smallmatrix} \sqrt{\sigma(\mu)} \\ 0 \\ 0\end{smallmatrix} \right)$.
Therefore,
$$
g \cdot \mu=\begin{pmatrix} 
\,1 & \,0 & 0 \\
\,0 & \,1 & \sqrt{\sigma(\mu)}\\
\,0 & -\sqrt{\sigma(\mu)} & 1 \end{pmatrix}. 
$$
If $\sigma(\mu)=\langle\,\text{\bf a}_\mu\,,\,\text{\bf a}_\mu\,\rangle=0$, 
we cannot make the choice 
$\text{\bf a}_\mu = {\sigma(\mu)}^{1/2}g_{\,i\,\ast}\,$
for any row $g_{\,i\,\ast}$ of a given orthogonal matrix $g$. 
However, we can still choose $g\in\GL_3(\Bbb C)$ to
bring $S_\mu$ into the alternative canonical form $S=S_{g\cdot\mu}
=\left(\begin{smallmatrix} K&0\\0&1\end{smallmatrix}
\right)$, with 
$K =\left(\begin{smallmatrix} 0&1\\1&0\end{smallmatrix}
\right)$. 
In fact, choose 
$g_{\,1\,\ast}=\text{\bf a}_\mu$
and complete a basis of $\Bbb C^3$ by
producing first a hyperbolic plane 
$\Pi=\operatorname{Span}\{\text{\bf a}_\mu,\text{\bf b}\}$,
and then choose a perpendicular vector to it:
$\text{\bf c}=z\,(\text{\bf a}_\mu\,\times\,\text{\bf b})$
($z\in\C-\{0\}$).
Thus, $\text{\bf b}$ satisfies,
$\langle \text{\bf b},\text{\bf b} \rangle=0$ and
$\langle \text{\bf a}_\mu,\text{\bf b}\rangle \ne 0$.
Therefore, there is a $g$ with row vectors
$\{ \text{\bf a}_\mu,\text{\bf b}, \text{\bf c}\}$
satisfying,
$g\,g^t=\left(
\begin{smallmatrix}
0 & \langle\,\text{\bf a}_\mu\,,\,\text{\bf b}\,\rangle & 0 \\
\langle\,\text{\bf a}_\mu\,,\,\text{\bf b}\,\rangle & 0 & 0 \\
0 & 0 & \langle\,\text{\bf c}\,,\,\text{\bf c}\,\rangle
\end{smallmatrix}
\right)$, 
and $\det g=z^{-1}\,\langle\,\text{\bf c}\,,\,\text{\bf c}\,\rangle$.
By choosing appropriate values of $z$ and $\det g$, we obtain
$g\in \operatorname{GL}_3(\C)$ with, $g\cdot \mu=\left(\begin{smallmatrix} \,0 & \,1 & \,0 \\
\,1 & \,0 & \!\!-1 \\
\,0 & \,1 & \,1 \end{smallmatrix} \right)$.
We thus have the following:

\begin{Prop}\label{multiplicaciones-casos-NO-degenerados} 
{\sl 
There are three different types of isomorphism classes of non-degenerate products
$\mu$ on $\g$, under the left $\GL(\g)$-action \eqref{GL-accion}. These are:
$$
\mu_{1;a}=\begin{pmatrix}
\,1 & \,0 & \,\,0 \\
\,0 & \,1 & \!-a \\
\,0 & \,a & \,\,1
\end{pmatrix}\!,\ a\ne 0;\qquad
\mu_2=1\!\!1_3;\qquad
\mu_3=\begin{pmatrix}
\,0 & \,1 & \,\,0 \\
\,1 & \,0 & \!-1 \\
\,0 & \,1 & \,\,1
\end{pmatrix}\!.
$$
Moreover, $\mu_{1;a}$ is equivalent to $\mu_{1;a^\prime}$
if and only if $a=\pm a^\prime\ne 0$\/.}
\end{Prop}

\begin{Remark}{\rm
By {\bf Prop. \ref{ProductosDeLie}}, none of the products $\mu_{1;a}$ 
nor $\mu_3$ define Lie algebra brackets on $\g$.
The only Lie algebra product corresponds to $\mu_2$,
which yields the Lie algebra $\mathfrak{sl}_2(\C)$.}
\end{Remark}

\section{Degenerate Products}

In this section we consider
those products whose symmetric component is {\it degenerate\/.}
We shall make use of the following well known result:

\begin{Lemma}\label{diagonalizacion}
{\sl Let $S \in \Mat_{3 \times 3}(\C)$ be a non-zero symmetric matrix with $\Rank(S)<3$. 
There exists 
$g\in \GL_3(\C)$, such that $g\,S\,g^t$ is equal to\/:}
\begin{enumerate}
\item {\sl $I_1:=\Diag \{1,0,0\}$ if and only if $\Rank(S)=1$\/;}
\item {\sl $I_2:=\Diag \{1,1,0\}$ if and only if $\Rank (S)=2$\/,}
\end{enumerate}
\end{Lemma}

Thus, if $\mu$ is a degenerate product, we may first assume that, 
$\mu=I_{\ell}+A_\mu$, where $1 \le \ell \le 2$. Then,
look at the isotropy subgroups,
$$
G_\ell :=\{g\in\GL_3(\C)\mid (\det g)^{-1} g\,I_\ell \, g^t = I_\ell\},\quad \ell =1,2.
$$
It is then a straightforward matter to see that:

\begin{Lemma}\label{grupos de isot} 
{\sl
The isotropy subgroups $G_\ell$ are given by\/,}
\begin{enumerate}
\item $G_1= \left\{ \begin{pmatrix} a & v^t \\ 0 & B \end{pmatrix}
\Big\vert
B\in\GL_2(\C),\ \det B=a \neq 0, \ v \in \C^2 \right\}$;
\item $G_2= \left\{  \begin{pmatrix}  B & v \\ 0 & \pm 1 \end{pmatrix}
\Big\vert\, B\in\GL_2(\C),\ 
BB^t=\pm (\det B)\,1\!\!1_2, \ v \in \C^2\right\}$,
\end{enumerate}
{\sl where $1\!\!1_2$ denotes the identity $2\times 2$ matrix\/.}
\end{Lemma}

\begin{Remark} 
{\rm
For some special cases discused in the proof of 
{\bf Prop.} \ref{degenerado rango 2} below,
different choices of canonical forms for $S_\mu$
will have to be made, as it was done 
in the case of $\mu_3$ in {\bf Prop.} \ref{multiplicaciones-casos-NO-degenerados},
and their corresponding isotropy subgroups
will have to be changed accordingly.
}
\end{Remark}

\begin{Prop}\label{degenerado rango 2}
{\sl
Let $\mu$ be a degenerate product of rank 2 and $A_\mu
\leftrightarrow \text{\bf a}_{\mu} \neq 0$.
Let $\pi_2(\text{\bf a}_{\mu})$ be the projection of $\text{\bf a}_{\mu}$ onto the $\C^2$
factor of $\C^3$ given by its first and second components
and let $\langle\,\cdot\,,\,\cdot\,\rangle_2$ be the usual scalar product on $\C^2$.
Then, there are three non-equivalent $\GL_3(\C)$-orbits, 
$\{g\cdot\mu\mid g\in\GL_3(\C)\}$,
described by the following conditions\/:}
\begin{enumerate}
\item {\sl $\langle \pi_2(\text{\bf a}_{\mu}),\pi_2(\text{\bf a}_{\mu}) \rangle_2 \neq 0 $,\/}
\item {\sl $\langle \pi_2(\text{\bf a}_{\mu}),\pi_2(\text{\bf a}_{\mu})\rangle_2=0$, 
         with $\pi_2(\text{\bf a}_{\mu}) \neq 0$, or,\/}
\item {\sl $\pi_2(\text{\bf a}_{\mu})=0$\/.}
\end{enumerate}
\end{Prop}

\begin{proof}
Let $A_\mu=\left(
\begin{smallmatrix} \,\,0&\!-z&\,\,\,y \\ 
\,\,z&\,\,\,0&\!-x \\ \!\!-y&\,\,\,x&\,\,\,0 \end{smallmatrix}
\right)\leftrightarrow\left(\begin{smallmatrix}x\\y\\z \end{smallmatrix}
\right)=\text{\bf a}_{\mu}$, with $\pi_2(\text{\bf a}_{\mu})=
\left(\begin{smallmatrix}x\\y \end{smallmatrix}\right) \in \C^2$.
Notice that $\langle\pi_2(\text{\bf a}_{g\cdot\mu}),\pi_2(\text{\bf a}_{g\cdot\mu})\rangle_2=
\pm \det B\,\langle\pi_2(\text{\bf a}_{\mu}),\pi_2(\text{\bf a}_{\mu})\rangle_2$.
Therefore,
either $\langle\pi_2(\text{\bf a}_{\mu}),\pi_2(\text{\bf a}_{\mu})\rangle_2\ne 0$
or $\langle\pi_2(\text{\bf a}_{\mu}),\pi_2(\text{\bf a}_{\mu})\rangle_2=0$, along the
orbits.
Suppose first that $\langle\pi_2(\text{\bf a}_{\mu}),\pi_2(\text{\bf a}_{\mu})\rangle_2 \neq 0$. 
We may choose $g\in G_2$ in such a way that,
$g_{1*}=\textbf{a}_{\mu}$, 
$g_{2*}=\left(\begin{smallmatrix}
-y\\ \,\,x \\ \,\,w
\end{smallmatrix}\right)$
and
$g_{3*}=\left(\begin{smallmatrix}
0 \\ 0 \\ 1
\end{smallmatrix}\right)$,
where $w\in\C$ can be chosen arbitrarily.
Then,
$\textbf{a}_{g \cdot \mu}=
\left(\begin{smallmatrix}
1 \\ 0 \\ 0
\end{smallmatrix}\right)$
and the canonical form of the product $\mu$ is, $g \cdot \mu=\left(\begin{smallmatrix}1 & \,0 & \,0 \\
0 & \,1 & \!\!-1 \\
0 & \,1  & \,0  \end{smallmatrix}\right)$. 

Now, suppose that $\langle \pi_2(\text{\bf a}_{\mu}),\pi_2(\text{\bf a}_{\mu}) \rangle_2=0$, 
with $\pi_2(\text{\bf a}_{\mu}) \neq 0$. 
Then, $\text{\bf a}_{\mu}$ cannot be embedded as a row of an element $g\in G_2$.
We may find however another vector $\text{\bf b}\in \C^3$ with
$\langle\pi_2(\text{\bf a}_{\mu}),\pi_2(\text{\bf b})\rangle_2=1$ and
$\langle\pi_2(\text{\bf b}),\pi_2(\text{\bf b})\rangle_2=0$ and
look for some $g\in\GL_3(\C)$ such that,
\begin{equation}\label{cambiar1por2aX}
\displaystyle{\frac{1}{\det g}}\,
\,g\,
\begin{pmatrix}
1\!\!1_2 & 0 \\
0 & 0
\end{pmatrix}
g^t =
\begin{pmatrix}
K & 0 \\
0 & 0
\end{pmatrix},\quad
\text{where,}\quad
K = 
\begin{pmatrix}
0 & 1 \\
1 & 0
\end{pmatrix}.
\end{equation}
Take,
$g_{1*}=\text{\bf a}_{\mu}=\left(\begin{smallmatrix} x \\ y \\ z\end{smallmatrix}\right)$,
$g_{2*}=\text{\bf b}=\left(\begin{smallmatrix} y\\ x \\ w \end{smallmatrix}\right)$,  
$g_{3*}=\left(\begin{smallmatrix} 0 \\ 0 \\ s \end{smallmatrix}\right)$, where, $x^2+y^2=0$.
Since
$g=\left(\begin{smallmatrix}
x & \,y & z\\
y & \,x & w \\
0 & \,0 & s
\end{smallmatrix}\right)$, $\det g=s(x^2-y^2)=2\,s\,x^2$.
Also, since $\pi_2(\text{\bf a}_{\mu}) \neq 0$ and $x^2+y^2=0$,
it follows that $y=\pm ix\ne 0$.
So, we may choose $s=y/x$ so as to satisfy 
\eqref{cambiar1por2aX}. 
Taking this $g$, it also follows
from \eqref{accion-vector-antisim2} that, $(\det g)^{-1}
g\,A_{\mu}g^t=\left(\begin{smallmatrix}
0 & \,0 & \,\,0 \\
0 & \,0 & \!-1 \\
0 & \,1 & \,\,0
\end{smallmatrix}\right)$. 
Therefore, the canonical form is, 
$g\cdot\mu = \left({\begin{smallmatrix}
0 & \,1 & \,\,0 \\
1 & \,0 & \!\!-1 \\
0 & \,1 & \,\,0
\end{smallmatrix}}\right)$,
whenever,
$\pi_2(\text{\bf a}_\mu)\ne 0$ and
$\langle\pi_2(\text{\bf a}_\mu),\pi_2(\text{\bf a}_\mu)\rangle_2=0$.

Finally, let $\pi_2(\text{\bf a}_\mu)=0$, with $\text{\bf a}_\mu\ne 0$.
Thus, $\text{\bf a}_\mu=\left(\begin{smallmatrix} 0\\0\\z\end{smallmatrix}\right)$, and $z\ne 0$.
We may choose $g_{3*}=
z^{-1}\text{\bf a}_{\mu}$,
and complete $g_{1*}$ and $g_{2*}$ so as to have $g=
\left(\begin{smallmatrix}  1\!\!1_2 & v \\ 0 & \pm 1 \end{smallmatrix}\right)
\in G_2$.
In particular,
$\text{\bf a}_{g \cdot \mu}=\left(\begin{smallmatrix}0 \\ 0 \\ 1 \end{smallmatrix}\right)$
and the canonical form of the product is in this case, $g \cdot \mu=\left(\begin{smallmatrix}1 & \!-1 & \,0 \\
1 & \,\,1 & \,0 \\
0 & \,\,0 & \,0  \end{smallmatrix} \right)$. 

We finally observe that the products represented by the matrices,
$\left(\begin{smallmatrix}
0 & \,1 & \,\,0 \\
1 & \,0 & \!-1 \\
0 & \,1 & \,\,0
\end{smallmatrix}\right)$
and
$\left(\begin{smallmatrix}
1 & \!-1 & \,0 \\
1 & \,\,1 &  \,0 \\
0 & \,\,0 & \,0
\end{smallmatrix}\right)$,
cannot be isomorphic.
If they were equivalent, there would be an element $g \in G_2$ such that
$g \!\cdot\!\!\left(\begin{smallmatrix}
\,\,1 & \!-z & \,\,y \\
\,\,z & \,\,1 & \!-x \\
\!-y & \,\,x & \,\,0
\end{smallmatrix}\right)=
\left(
\begin{smallmatrix}
1 & \!-1 & \,0 \\
1 & \,\,1 & \,0 \\
0 & \,\,0 & \,0
\end{smallmatrix}\right)$,
thus implying that $(x,y)=0$, which is a contradiction.
\end{proof}

We shall now deal with the degenerate products of rank 1.

\begin{Prop}\label{degenerado rango 1}
{\sl 
Let $\mu=I_1+A_\mu$ be a degenerate product of rank 1 and $A_\mu
\leftrightarrow \text{\bf a}_{\mu} \neq 0$.
Let $\pi_1(\text{\bf a}_{\mu})$ be the projection of $\text{\bf a}_{\mu}$ onto 
its first component in $\C$.
Then, there are two non-equivalent $G_1$-orbits described by the following conditions: either\/,}
\begin{enumerate}
\item {\sl $\pi_1(\text{\bf a}_{\mu})\ne 0$, or else,}
\item {\sl $\pi_1(\text{\bf a}_{\mu})=0$\/.}
\end{enumerate}
\end{Prop}
\begin{proof}
Let $A_\mu=\left(
\begin{smallmatrix} \,\,0&\!-z&\,\,\,y \\ 
\,\,z&\,\,\,0&\!-x \\ \!\!-y&\,\,\,x&\,\,\,0 \end{smallmatrix}
\right)\leftrightarrow\left(\begin{smallmatrix}x\\y\\z \end{smallmatrix}
\right)=\text{\bf a}_{\mu}$, with $\pi_1(\text{\bf a}_{\mu})=x \in \C$.
Take $g=\left(\begin{smallmatrix} 
a & v^t \\ 0 & B \end{smallmatrix}\right) \in G_1$, so that
$\det g = a\,\det B$.
Then,
$\pi_1(\text{\bf a}_{g\cdot\mu})=
(\det g)^{-1}\langle \,\text{\bf a}_\mu\times g_{\,2\,\ast}\,,\,g_{\,3\,\ast}\, \rangle 
=a^{-1}\pi_1(\text{\bf a}_\mu)$.
Whence, either the first component in the orbit is zero, 
or it is different from zero.
Assume first that $x=\pi_1(\text{\bf a}_\mu)\ne 0$. 
Let $g\in G_1$ be given by,
$g_{1*}=\textbf{a}_{\mu}$,
$g_{2*}=\left(\begin{smallmatrix}
0\\ \sqrt{x} \\ 0
\end{smallmatrix}\right)$, and
$g_{3*}=\left(\begin{smallmatrix}
0 \\ 0 \\ \sqrt{x}
\end{smallmatrix}\right)$, where, $x\ne 0$.
Then, $\text{\bf a}_{g \cdot \mu}
=\left( \begin{smallmatrix}1 \\ 0 \\ 0\end{smallmatrix}\right)$, 
and the canonical form of the product is, 
$g \cdot \mu=\left(\begin{smallmatrix}1 & \,0 & \,\,0 \\ 
0 & \,0 & \!-1 \\ 
0 & \,1 & \,\,0   \end{smallmatrix}\right)$. 
On the other hand, if
$x=0$, but $\text{\bf a}_{\mu}\ne 0$, the 
vector $\textbf{a}_{\mu}$ can be embedded into a row of some $g\in G_1$; 
either as $g_{2*}= \text{\bf a}_{\mu}$, or else, as 
$g_{3*}=\text{\bf a}_{\mu}$.
It is clear that any of these choices will not change the orbit. Thus, choose
$g_{2*}=\textbf{a}_{\mu}$,
and complete $g_{1*}$ and $g_{3*}$ so as to get $g\in G_1$.
It is a straightforward matter to see that in this case, $g \cdot \mu=\left(\begin{smallmatrix}\,\,1 & 0 & \,1 \\
\,\,0 & 0 & \,0 \\
\!-1 & 0 & \,0   \end{smallmatrix}\right)$. 
Observe that the two representatives thus found cannot be
equivalent, as the latter is non-invertible, whereas the former is.
\end{proof}

We may now summarize our findings for degenerate products 
$\mu$ with non-zero symmetric component $S_\mu$.

\begin{Prop}\label{multiplicaciones-casos degenerados} 
{\sl
Let $\mu^{\ell}$ be a degenerate product on $\g$
with non-zero symmetric component $S_\mu$ and 
$\ell=\operatorname{rk}S_\mu$.
Also write $\ell=\prime\prime$ for rank 2, and $\ell=\prime$ for rank 1\/.}
\begin{enumerate}
\item  
{\sl 
If $\operatorname{rk}S_\mu=2$, there are four non-equivalent products\/:}
$$
\mu^{\prime\prime}_1:= 
\begin{pmatrix}
1&\,0&\,\,0 \\
0&\,1&\!-1 \\
0&\,1&\,\,0 
\end{pmatrix}, \quad
\mu^{\prime\prime}_2:= 
\begin{pmatrix}
1&0&0 \\
0&1&0 \\
0&0&0 
\end{pmatrix}, \quad
\mu^{\prime\prime}_3:= 
\begin{pmatrix} 
0&\,1&\,\,0 \\ 
1&\,0& \!-1 \\ 
0&\,1&\,\,0 \end{pmatrix},
$$
$$
\mu^{\prime\prime}_4:=
\begin{pmatrix}
\,\,1&1&\,0  \\
\!-1&1&\,0 \\
\,\,0&0&\,0
\end{pmatrix}.
$$
\item  
{\sl If $\operatorname{rk}S_\mu=1$, there are three non-equivalent products\/:}
$$
\mu^{\prime}_1:=
\begin{pmatrix}
1&\,\,0&\,0 \\ 
0&\,\,0&\,1 \\ 
0&\!-1&\,0
\end{pmatrix}, \quad
\mu^{\prime}_2:= 
\begin{pmatrix}
1&0&0 \\
0&0&0 \\
0&0&0
\end{pmatrix}, \quad
\mu^{\prime}_3:= 
\begin{pmatrix}
\,\,1&0&\,1 \\
\,\,0&0&\,0 \\
\!-1&0&\,0 \end{pmatrix}.
$$
\item
{\sl
If  $\operatorname{rk}S_\mu=0$ and $\mu \ne 0$,
then the product representative is:\/}
$$
 \mu_0=
\begin{pmatrix} 0 & \,0 & \,\,0 \\ 0 & \,0 & \!-1 \\ 0 & \,1 & \,\,0
\end{pmatrix}.
$$
\end{enumerate}

\end{Prop}

\begin{Remark} 
{\rm
Once again, 
{\bf Prop. \ref{ProductosDeLie}} implies that neither $\mu_1^{\prime\prime}$ nor $\mu_1^{\prime}$,
can be representatives of Lie algebra brackets
on $\g$. When the matrix $\mu$ is singular, one must verify directly whether
$\mu$ defines or not a Lie algebra bracket. It can be checked that
$\mu_4^{\prime\prime}$, $\mu_2^{\prime\prime}$, 
$\mu_2^{\prime}$ and $\mu_0$ are Lie brackets,
whereas $\mu_3^{\prime\prime}$ and
$\mu_3^{\prime}$ are not.
Moreover, the Lie algebras
$(\g,\mu_4^{\prime \prime})$, $(\g,\mu_2^{\prime \prime})$ and $(\g,\mu_0)$ 
are non-nilpotent and solvable, whereas $(\g,\mu^{\prime}_2)$
is the Heisenberg Lie algebra.
}
\end{Remark}

In the following two sections we shall 
determine the $G_{\mu}$-orbits in 
$\operatorname{HL}(\mu)$, 
under the left action $T\mapsto g\cdot T = g\,\,T\,g^{-1}$, where,
$$
G_{\mu}=\{g \in \GL(\g)\,|\,g\,(\mu\,(g^{-1}(x),g^{-1}(y)))=\mu(x,y),\,\forall\, x,y \in \g\},
$$  
is the isotropy subgroup at the given representative $\mu$; 
{\it ie\/, the automorphism group of the canonical form\/} $\mu$.

\section{Hom-Lie algebras associated to non-degenerate products}

We shall start with the
three different  isomorphism classes of non-degenerate products
in a 3-dimensional vector space $\g$
given by {\bf Prop.} \ref{multiplicaciones-casos-NO-degenerados}.

\subsection{The one-parameter family $\mu_{1;a}$\,}
Let $\{e_1,e_2,e_3 \}$ be a basis of $\g$ for which
the matrix associated to $\mu_{1;a}$ is,
$\left(\begin{smallmatrix}
1 & \,\,0 & \,0 \\
0 & \,\,1 & \!-a \\
0 & \,\,a & \,1
\end{smallmatrix}\right)$, where $a\ne 0$.
Clearly, $\det \mu_{1;a}=1+a^2$.
Thus, the cases  $a =\pm i$ must be considered separately.
Let $\{H^{\prime},E^{\prime},F^{\prime}\}$ be the basis defined by,
$$
H^{\prime}=-2i e_1\,,\quad E^{\prime}=e_2-ie_3\,,\quad F^{\prime}=-e_2-ie_3,
$$
so that,
$\mu_{1;a}(H^{\prime},E^{\prime})=2(1+ia)\,E^{\prime}$,
$\mu_{1;a}(H^{\prime},F^{\prime})=-2(1-ia)\,F^{\prime}$ and 
$\mu_{1;a}(E^{\prime},F^{\prime})=H^{\prime}$.
In this basis, $\mu_{1;a}\leftrightarrow
\left(\begin{smallmatrix} 
1 & \,\,0 & 0 \\
0 & \,\,0 & 2(1+i a)  \\
0 & \,\,2(1-i a) & 0
\end{smallmatrix}\right)$.
Clearly, if $a \neq \pm i$, then $\mu_{1;a}(\g,\g)=\g$ and the product is {\it perfect\/.}
However, $\mu_{1;\pm i}(\g,\g) \neq \g$.

\begin{Prop}\label{ND-1}
{\sl 
Let $\mu_{1;a}:\g \times \g \rightarrow \g$, $a \in \C-\{0\}$, be the non-degenerate product defined on the basis $\{e_1,e_2,e_3\}$ by,
$$
\mu_{1;a}(e_1,e_2)=-ae_2+e_3,\quad \mu_{1;a}(e_2,e_3)=e_1, \quad \mu_{1;a}(e_3,e_1)=e_2+ae_3.
$$
{\bf A.} If $a\ne \pm i$,
the HL algebra $(\g,\mu_{1;a},T)$ is perfect and $T=(T_{ij}) \in \operatorname{HL}(\mu_{1;a})$ is equivalent to one and only one of the following canonical forms\/:}
\begin{enumerate}

\item 
{\sl 
If $T_{12}^2+T_{13}^2 \neq 0$, then\/,}
$$
T \simeq \begin{pmatrix}
T^{\prime}_{11} & 1 & 0 \\
1 & T^{\prime}_{22} & T^{\prime}_{23} \\
a & a(T^{\prime}_{22}+T^{\prime}_{33})+T^{\prime}_{23} & T^{\prime}_{33}
\end{pmatrix}\!,
$$

\item
{\sl 
If $T_{12}^2+T_{13}^2=0$ and $(T_{12},T_{13})\neq (0,0)$, then\/,}
$$
T \simeq \begin{pmatrix}
T^{\prime}_{11} & 1 & i \\
1-ai & T^{\prime}_{22} & 0 \\
i(1-ai) & a(T^{\prime}_{22}+T^{\prime}_{33}) & T^{\prime}_{33}
\end{pmatrix}\!,
$$

\item
{\sl 
If $(T_{12},T_{13})=(0,0)$, then\/,}
$$
T \simeq \begin{pmatrix}
T^{\prime}_{11} & 0 & 0 \\
0 & T^{\prime}_{22} & 0 \\
0 & a(T^{\prime}_{22}+T^{\prime}_{33}) & T^{\prime}_{33}
\end{pmatrix}\!.
$$
\end{enumerate}
\smallskip
{\bf B.} 
{\sl 
If $a=\pm i$, the products $\mu_{1;i}$ and $\mu_{1;-i}$ are equivalent and
there is a basis of $\g$ in terms of which their corresponding matrices are equivalent to\/:}
$$
\mu_{1;\pm i}:=
\begin{pmatrix}
1 & 0 & 0 \\
0 & 0 & 1 \\
0 & 0 & 0 
\end{pmatrix}.
$$
{\sl 
In this case, the HL algebra $(\g,\mu_{1;\pm i},T)$ is not perfect
and $T=(T_{ij}) \in \operatorname{HL}(\mu_{1; \pm i})$
if and only if in this basis the matrix entries of $T$ 
satisfy, $T_{21}=T_{13}$ and $T_{33}=0$.
Moreover, any $g\in G_{\mu_{1; \pm i}}$ is of the form
$g=\operatorname{diag}(1,\lambda,\lambda^{-1})$, with $\lambda\ne 0$,
and\/,}
$$
T^\prime=g\,T\,g^{-1}=\begin{pmatrix}
T_{11} & \lambda^{-1}T_{12} & \lambda\,T_{13} \\
\lambda\,T_{13} & T_{22} & \lambda^2T_{23} \\
\lambda^{-1}T_{31} & \lambda^{-2}T_{32} & 0
\end{pmatrix}.
$$
\end{Prop}

\begin{proof}
{\bf A.}
For $a \neq \pm i$, we shall work in the basis $\{e_1,e_2,e_3\}$.
Any $g\in G_{\mu_{1;a}}$ in this basis has the form
$\left(\begin{smallmatrix}
1 & \,0 & 0\\
0 & \,x & y\\
0 & \!\!-y & x
\end{smallmatrix}\right)$, with $x^2+y^2=1$. 
First observe that, $T \in \operatorname{HL}(\mu_{1;a})$, 
if and only if its matrix has the form:
$$
T=\begin{pmatrix}
T_{11} & T_{12} & T_{13} \\
T_{12}-a T_{13} & T_{22} & T_{23} \\
a T_{12}+T_{13} & a(T_{22}+T_{33})+T_{23} & T_{33}
\end{pmatrix}.
$$
Write $T^\prime=(\,T^\prime_{ij})$ for the matrix defined by
$T^\prime=g\,\,T\,g^{-1}$ with $g\in G_{\mu_{1;a}}$.
It is easy to see that,
$$
\aligned
T_{11}^{\prime}=&T_{11} \\
T^{\prime}_{12}=& T_{12} x+T_{13} y, \\
T^{\prime}_{13}=&-T_{12} y+T_{13} x, \\
T^{\prime}_{22}=&T_{22} x^2+(a(T_{22}+T_{33})+2T_{23})xy+T_{33} y^2,\\
T^{\prime}_{23}=&T_{23} x^2-(a(T_{22}+T_{33})+T_{23})y^2+(T_{33}-T_{22})xy,\\
T^{\prime}_{33}=&T_{33} x^2-(a(T_{22}+T_{33})+2T_{23})xy+T_{22} y^2.
\endaligned
$$
It follows that ${T^\prime_{12}}^2+{T^\prime_{13}}^2
={T_{12}}^2+{T_{13}}^2$. Suppose $(T_{12},T_{13}) \neq (0,0)$. 
If ${T_{12}}^2+{T_{13}}^2\ne 0$, 
we may choose $x$ and $y$ in such a way that $T$ becomes 
equivalent to the canonical form,
$$
T^\prime = \begin{pmatrix}
T_{11}^{\prime} & 1 & 0 \\
1 & T_{22}^{\prime} & T_{23}^{\prime} \\
a & a(T_{22}^{\prime}+T_{33}^{\prime})+T_{23}^{\prime} & T_{33}^{\prime}
\end{pmatrix},
\quad
\begin{matrix}
(T_{12},T_{13}) \neq (0,0),
\\
\text{and}
\\
{T_{12}}^2+{T_{13}}^2\ne 0.
\end{matrix}
$$
On the other hand, if $(T_{12},T_{13}) \neq (0,0)$, but
${T_{12}}^2+{T_{13}}^2=0$, we may assume that 
$T_{13}=i T_{12}\ne 0$ and 
appropriate choices of $x$ and $y$ in $g$ bring
$T^\prime=g\,T\,g^{-1}$, to the following canonical form:
$$
T^\prime=
\begin{pmatrix}
T^{\prime}_{11} & 1 & i \\
1-ai & T^{\prime}_{22} & 0 \\
i(1-a_i) & a(T^{\prime}_{22}+T^{\prime}_{33}) & T^{\prime}_{33}
\end{pmatrix}.
$$
Finally, for the case $T_{12}=T_{13}=0$, 
there is a $g\in G_{\mu_{1;a}}$ such that
$T^\prime=g\,T\,g^{-1}$, is brought to the following canonical form:
$$
\begin{pmatrix}
T_{11} & 0 & 0 \\
0 & T_{22} & 0 \\
0 & a(T_{22}+T_{33}) & T_{33}
\end{pmatrix}.
$$
{\bf B.}
For $a = \pm i$, we shall work in the basis $\{H^\prime,E^\prime,F^\prime\}$.
If $a=i$. the matrix of the product in this basis takes the form,
$\left(
\begin{smallmatrix} 
1 & 0 & 0 \\
0 & 0 & 0  \\
0 & 4 & 0
\end{smallmatrix}\right)$.
On the other hand, if 
$a=-i$, the matrix of the product has the form,
$\left(
\begin{smallmatrix} 
1 & 0 & 0 \\
0 & 0 & 4  \\
0 & 0 & 0
\end{smallmatrix}\right)$.
These two matrices are equivalent to
$\left(
\begin{smallmatrix} 
1 & 0 & 0 \\
0 & 0 & 1  \\
0 & 0 & 0
\end{smallmatrix}\right)$.
Thus, there is a basis $\{e''_1,e''_2,e''_3\}$ of $\g$ 
for which,
$$
\mu_{1:\pm i}(e''_1,e''_2)=e''_3,\quad \mu_{1;\pm i}(e''_2,e''_3)=e''_1,\quad \mu_{1;\pm i}(e''_3,e''_1)=0.
$$
It is easy to see that its isotropy subgroup is given as in the statement.
It is also easy to verify that 
$(T_{ij})=T\in\operatorname{HL}(\mu_{1;\pm i})$ if and only if
$T_{21}=T_{13}$ and $T_{33}=0$.
Then, for any $g\in G_{\mu_{1;\pm i}}$ the matrix
$T^\prime=g\,T\,g^{-1}$, takes the form:
$$
T^\prime=\begin{pmatrix}
T_{11} & \lambda^{-1}T_{12} & \lambda\,T_{13} \\
\lambda\,T_{13} & T_{22} & \lambda^2T_{23} \\
\lambda^{-1}T_{31} & \lambda^{-2}T_{32} & 0
\end{pmatrix}.
$$
Observe that no further simplification can be made, except
for rescaling some entries; say, if $T_{13}\ne 0$, we may choose
$\lambda$ so that $T_{13}=1$, but then this choice
fixes the scaling of all the other off-diagonal entries.
\end{proof}

\subsection{HL-algebras for the Lie product of $\mathfrak{sl}_2(\C)$}

Let $\mu_2$ 
be the Lie bracket of $\g=\mathfrak{sl}_2(\C)$. 
We shall use the basis $\{H,E,F\}$ for which, $\mu_2(H,E)=2E$, $\mu_2(E,F)=H$ 
and $\mu_2(H,F)=-2F$. 
The isotropy group
$G_{\mu_2}=\{g\in\operatorname{GL}(\frak{sl}_2(\C))\mid g\!\cdot\!\mu_2=\mu_2\}$
coincides with $\Aut\mathfrak{sl}_2(\C)$ and its structure is well known
(see for example, \cite{Hum}). It is a straightforward matter to see that
any $T\in \operatorname{HL}(\mu_2)$ in the basis $\{H,E,F\}$
can be written in the form,
\begin{equation*}\label{descomposicion de HL}
\begin{pmatrix}
\displaystyle{\frac{2(T_{11}-T_{22})}{3}} & T_{12} & T_{13} \\
2 T_{13} & \!\!-\displaystyle{\frac{T_{11}-T_{22}}{3}} & T_{23} \\
2 T_{12} & T_{32} & \!\!-\displaystyle{\frac{T_{11}-T_{22}}{3}}
\end{pmatrix}
\!+\frac{1}{3}(T_{11}+2T_{22})\operatorname{Id}_{\frak{sl}_2(\C)}.
\end{equation*}
This suggests to decompose $\operatorname{HL}(\mu_2)$ as  
$\operatorname{SHL}(\mu_2) \oplus \C \operatorname{Id}_{\mathfrak{sl}_2(\C)}$, where,
$\operatorname{SHL}(\mu_2) :=
\{T \in \operatorname{HL}(\mathfrak{sl}_2(\C))\,|\,\operatorname{Tr}(T)=0\}$.
Write any $T\in\operatorname{HL}(\mu_2)$ in the form $T=T_0+\lambda\,
\operatorname{Id}_{\frak{sl}_2(\C)}$, with $T_0\in\operatorname{SHL}(\mu_2)$.
If $T^\prime=T^\prime_0+\lambda^\prime\,
\operatorname{Id}_{\frak{sl}_2(\C)}$, then
$T$ is equivalent to $T^{\prime}$ under the $G_{\mu_2}$-action
$T\mapsto T^\prime=g\,\,T\,g^{-1}$
if and only if $T_0$ is equivalent to $T^{\prime}_0$ and $\lambda=\lambda^{\prime}$. 
Therefore, this leads us to determine
first the $G_{\mu_2}$-orbits in the subspace $\operatorname{SHL}(\mu_2)$. 
It has been proved in \cite{Garcia} that 
$\operatorname{SHL}(\mu_2)
=\operatorname{Der}_{(-1,1,1)}(\mathfrak{sl}_2(\C))$,
the latter being the space of 
{\it generalized derivations of
type\/} $(-1,1,1)$. This means that any $T \in \operatorname{SHL}(\mu_2)$ 
satisfies the following form of  {\it Leibniz's rule\/:}
$$
-T(\mu_2(x,y))=\mu_2(T(x),y)+\mu_2(x,T(y)),\quad \forall x,y \in \mathfrak{sl}_2(\C).
$$
The $G_{\mu_2}$-orbits in $\operatorname{Der}_{(-1,1,1)}(\mathfrak{sl}_2(\C))$
have been determined in \cite{Garcia} 
and using the canonical forms given there
({\bf Thm.} 3.6 in \cite{Garcia}), one may easily prove the following: 

\begin{Prop}\label{ND-2}
Let $\mu_2$ be the Lie bracket on $\g=\mathfrak{sl}_2(\C)$ and 
let $G_{\mu_2}=\{g\in\operatorname{GL}(\frak{sl}_2(\C))\mid g\cdot\mu_2=\mu_2\}
=\Aut \frak{sl}_2(\C)$
be its isotropy group.
Let $T_0=T-1/3\operatorname{Tr}(T)  \operatorname{Id}_{\mathfrak{sl}_2}$,
for each $T\in \operatorname{HL}(\mu_2)$.
Then, the $G_{\mu_2}$-action $T\mapsto 
T^{\prime}=g\,\, T\, g^{-1}$ decomposes 
$\operatorname{HL}(\mu_2)$ into different orbits whose representatives
(canonical forms) are given by:
$$
 \text{If}\ 
T_0=0
\ \Rightarrow\ 
T\sim \begin{pmatrix}
T^{\prime}_{11} & 0 & 0 \\
0 & T^{\prime}_{11} & 0\\
0 & 0 & T^{\prime}_{11}
\end{pmatrix},\quad \text{with}\quad \aligned T^{\prime}_{11} & \in  \C,\endaligned 
$$
$$
 \text{If}\ 
\operatorname{rank}(T_0)=1
\ \Rightarrow\ 
T\sim \begin{pmatrix}
T^{\prime}_{11} & 0 & 0 \\
0 & T^{\prime}_{11} & 1 \\
0 & 0 & T^{\prime}_{11}
\end{pmatrix},\quad \text{with}\quad \aligned T^{\prime}_{11} & \in  \C,
\endaligned 
$$ 
$$
 \text{If}\ 
\operatorname{rank}(T_0)=2
\ \Rightarrow\ 
T\sim
\begin{cases}
{
\begin{pmatrix}
T^{\prime}_{11} & 0 & 1 \\
2 & T^{\prime}_{11} & 0 \\
0 & 0 & T^{\prime}_{11}
\end{pmatrix},
} 
& \text{with}\quad \aligned T^{\prime}_{11} & \in \C,\endaligned
\\
{\ \ }& {\ }
\\
{
\begin{pmatrix}
T^{\prime}_{11} & 1 & T^{\prime}_{13} \\
2T^{\prime}_{13} & T^{\prime}_{11} & 0 \\
2 & 0 & T^{\prime}_{11}
\end{pmatrix},
}
&
 \text{with}\quad \aligned T^{\prime}_{11} & \in  \C,\\ 
T^{\prime}_{13}
& \ne
0
\endaligned
\end{cases}
$$
$$
 \text{If}\ 
\operatorname{rank}(T_0)=3
\ \Rightarrow\ 
T\sim\begin{cases}
{
\begin{pmatrix}
T^{\prime}_{11} & 1 & 0 \\
0 & T^{\prime}_{11} & T^{\prime}_{23} \\
2 & 0 & T^{\prime}_{11}
\end{pmatrix},
} 
& \text{with}\quad \aligned T^{\prime}_{11} & \in \C,\\ T^{\prime}_{23} & \ne 0; \endaligned 
\\
{\ \ }& {\ }
\\
{
\begin{pmatrix}
T^{\prime}_{11} & 1 & T^{\prime}_{13} \\
2T^{\prime}_{13} & T^{\prime}_{11} & T^{\prime}_{23} \\
2 & 0 & T^{\prime}_{11}
\end{pmatrix},
}
&
 \text{with}\,\, \aligned T^{\prime}_{11} & \in \C,\\ T^{\prime}_{13}T^{\prime}_{23} & \neq 0; \endaligned
\end{cases}
$$
\end{Prop}

\begin{Remark}{\rm
It is proved from first principles
in \cite{Garcia} that $\mathfrak{sl}_2(\C)$ is the only
simple Lie algebra that admits non-trivial HL-structures.
A previous software-assisted proof of this fact was given in \cite{Xie}
where HL structures on $\mathfrak{sl}_2(\C)$ were first studied. 
}
\end{Remark}

\begin{Remark}{\rm
Since $\mu_2$ is a Lie bracket, {\bf Cor. \ref{Gil}} applies, and
the bilin\-ear form $B_\mu$ there is actually a scalar multiple of the
Cartan-Killing form in $\frak{sl}_2(\C)$. Using the latter,
it is a straightforward matter to show that the the self-adjoint operators $T$ 
fit precisely into the canonical forms for the twist maps $T$
given in {\bf Prop. \ref{ND-2}}.}
\end{Remark}

\subsection{HL-algebras for the product $\mu_3$}

We have proved that there is a basis of $\g$ for which
the product $\mu_3$ has the matrix,
$\left(
\begin{smallmatrix}
0 & \,\,1 & \,\,0 \\
1 & \,\,0 & \!-1 \\
0 & \,\,1 & \,\,1
\end{smallmatrix}\right)$.
It can be decomposed into its symmetric and skew-symmetric components
as, $\mu_3 = S_{\mu_3} +A_{\mu_3}$. 
It is not difficult to see that 
$g \in \GL(\g)$ satisfies $g\cdot S_{\mu_3}=S_{\mu_3}$ and 
$g \cdot A_{\mu_3}=A_{\mu_3}$ if and only if $g=1\!\!1_3$; that is, the isotropy 
subgroup $G_{\mu_3}$ consists only of the identity map. 
It is also easy to see that there is a basis 
$\{e^{\prime}_1,e^{\prime}_2,e^{\prime}_3\}$, satisfying,
$$
\mu_3(e^{\prime}_1,e^{\prime}_2)=-e^{\prime}_1+e^{\prime}_3,\quad \mu_3(e^{\prime}_2,e^{\prime}_3)=e^{\prime}_2+e^{\prime}_3,\quad \mu_3(e^{\prime}_3,e^{\prime}_1)=e^{\prime}_1.
$$
Now let $H^{\prime\prime}=2e^{\prime}_3$, $E^{\prime\prime}=4e^{\prime}_1$ and $F^{\prime\prime}=\frac{1}{2}(e^{\prime}_2+e^{\prime}_3)$. 
Were not for its $(2,1)$-entry, the matrix of $\mu_3$ 
in the basis $\{H^{\prime \prime},E^{\prime \prime},F^{\prime \prime}\}$ 
looks almost like the matrix used for $\mu_2$ above. In fact,
\begin{equation}\label{producto excepcional}
\mu_3=
\begin{pmatrix}
\,\,1 & 0 & \,0 \\
\!-1 & 0 & \,2 \\
\,\,0 & 2 & \,0
\end{pmatrix}.
\end{equation}
A straightforward argument now proves the following:

\begin{Prop}\label{ND-3}
{\sl Fix the basis of $\g$ so that the non-degenerate product
$\mu_3$ takes the form \eqref{producto excepcional}. 
Then, $T \in \operatorname{HL}(\mu_3)$ if and only if, its matrix has the form:
$$
T=\begin{pmatrix}
T_{11} & \,T_{12} & \,\,T_{13}\\
2T_{13}-T_{11}-3T_{12} & \,T_{22} &\,\, T_{23} \\
2\,T_{12} & \,T_{32} & \,\,T_{22}-T_{12}
\end{pmatrix}\!.
$$\/}
\end{Prop}

\section{Hom-Lie algebras associated to degenerate products}

\subsection{HL-algebras for degenerate products of rank 2}

\subsubsection{HL-algebras for $\mu^{\prime\prime}_1$\,}

Let $\{e_1,e_2,e_3\}$ be the basis for which
$\mu^{\prime\prime}_1=\left(
\begin{smallmatrix}
1 & \,0 & \,\,0 \\
0 & \,1 & \!-1 \\
0 & \,1 & \,\,0
\end{smallmatrix}\right)$.
The isotropy subgroup $G_{\mu^{\prime\prime}_1}$ is given by:
$$
G_{\mu^{\prime \prime}_1}=\left\{\,
\begin{pmatrix}
1 & 0 \\
0 & \Lambda_{\pm}
\end{pmatrix}
\Bigg\vert\ \,\Lambda_{\pm} = \begin{pmatrix}
\pm1 & a \\
0 & \pm1
\end{pmatrix},\,a \in \C
\right\}.
$$
Then $T \in \operatorname{HL}(\mu^{\prime\prime}_1)$, if and only if it has the matrix form, $T=\left(\begin{smallmatrix}T_{11} & w^t \\
Pw & \Theta  \end{smallmatrix}\right)$, where $w=\left(\begin{smallmatrix}T_{12} \\T_{13}  \end{smallmatrix}\right)$, $\Theta=\left(\begin{smallmatrix} T_{22} & \,\,T_{23} \\
\!-(T_{22}+T_{33}) & \,\,T_{33} \end{smallmatrix}\right)$, and $P=\left(\begin{smallmatrix} 1 & 1 \\ 0 & 1\end{smallmatrix} \right)$.
Now, upon conjugation by $g_{\pm}\in G_{\mu^{\prime\prime}_1}$, we get:
$$
g_{\pm}T\,g_{\pm}^{-1}=
\begin{pmatrix}
T_{11} & w^t \Lambda_{\pm}^{-1} \\
\Lambda_{\pm}Pw & \Lambda_{\pm}\Theta \Lambda_{\pm}^{-1}
\end{pmatrix}\!,
$$
where,
$$
\Lambda_{\pm}\Theta \Lambda_{\pm}^{-1}  
 =
\begin{pmatrix}
T_{22}\mp a\,\Tr(\Theta) & a^2\Tr(\Theta)\mp a\,(\Tr(\Theta)-2T_{33})+T_{23} \\
-\Tr(\Theta) & T_{33}\pm a\,\Tr(\Theta))
\end{pmatrix}.
$$
If $\Tr(\Theta) \neq 0$ or $\Tr(\Theta) \neq  2T_{33}$, we may choose $a \in \mathbb{C}$
in such a way that, $a^2\Tr(\Theta) \mp a\,(\Tr(\Theta)-2T_{33})+T_{23}=0$, thus bringing $\Lambda_{\pm}\Theta \Lambda_{\pm}^{-1}$ into a lower triangular form.
Therefore, we have the following:

\begin{Prop}\label{D2-1}
{\sl Let $\{e_1,e_2,e_3\}$ be the basis of $\g$
with respect to which the degenerate product $\mu^{\prime\prime}_1$ 
is given by:
$$
\mu^{\prime\prime}_1(e_1,e_2)=e_2,\quad \mu^{\prime\prime}_1(e_2,e_3)=e_1,\quad \mu^{\prime\prime}_1(e_3,e_1)=e_2-e_3.
$$
Then,
$T=(T_{ij})\in \operatorname{HL}(\mu^{\prime\prime}_1)$ is equivalent to a one and only one of the following canonical forms\/:}
\begin{enumerate}

\item {\sl If $T_{22}\neq T_{33}$ or $T_{22}\neq -T_{33}$, then\/}
$$
T \simeq \begin{pmatrix}
T^{\prime}_{11} & T^{\prime}_{12} & T^{\prime}_{13} \\
T^{\prime}_{12}+T^{\prime}_{13} & T^{\prime}_{22} & 0 \\
T^{\prime}_{13} & -(T^{\prime}_{22}+T^{\prime}_{33}) & T^{\prime}_{33}
\end{pmatrix}\!,\quad \mbox{ with }T^{\prime}_{22}\neq T^{\prime}_{33} \mbox{ or }T^{\prime}_{22}\neq -T^{\prime}_{33}.
$$
\item {\sl If $T_{22}=T_{33}=0$, then \/}
$$
T \simeq \begin{pmatrix}
T^{\prime}_{11} & T^{\prime}_{12} & T^{\prime}_{13} \\
T^{\prime}_{12}+T^{\prime}_{13} & 0 & T^{\prime}_{23} \\
T^{\prime}_{13} & 0 & 0
\end{pmatrix}\!.
$$

\end{enumerate}
\end{Prop}

\subsubsection{HL-algebras for the product $\mu^{\prime\prime}_2$.}

Let $\{e_1,e_2,e_3\}$ be the basis of $\g$
with respect to which
$
\mu^{\prime\prime}_2=
\left(\begin{smallmatrix}
1 & 0 & 0 \\
0 & 1 & 0 \\
0 & 0 & 0
\end{smallmatrix}\right)$.
It is a straightforward matter to show that the
isotropy subgroup $G_{\mu^{\prime\prime}_2}$ is given by,
$$
G_{\mu^{\prime\prime}_2} = \left\{
g=
\begin{pmatrix}
\Lambda & u \\ 0 & \pm 1
\end{pmatrix}\Big\vert
\, u\in\Bbb C^2,\ \text{and}\ \Lambda\in G \,
\right\},
$$
$$
\text{where,}\quad
G = \left\{
\begin{pmatrix}
a & - b \\ b & a
\end{pmatrix} \Big\vert\,
a^2+b^2\ne 0
\right\} 
\cup \left\{
\begin{pmatrix}
a &  b \\ b & -a
\end{pmatrix} \Big\vert\,
a^2+b^2\ne 0
\right\}. 
$$
Clearly, $G\simeq (\Bbb C-\{0\})\times O_2(\Bbb C)$.
On the other hand, it is also a straightforward matter to verify that
$T \in \operatorname{HL}(\mu^{\prime\prime}_2)$ if and only if
its matrix $(T_{ij})$ satisfies $T_{13}=T_{23}=0$. Thus,
\begin{equation}\label{TenH(mu2)}
T=\begin{pmatrix}
\Theta & v \\
0 & T_{33}
\end{pmatrix}\!,\quad \Theta
=\begin{pmatrix}
T_{11} & T_{12} \\
T_{21} & T_{22}
\end{pmatrix},\ \ v=\begin{pmatrix}
T_{13} \\ T_{23}
\end{pmatrix}.
\end{equation}
Conjugation of $T$ by an element $g\in G_{\mu^{\prime\prime}_2}$
produces an equivalent $T^\prime$, given by,
$$
T^{\prime}=\begin{pmatrix}
\Theta^{\prime} & v^{\prime}\\
0 & T_{33}
\end{pmatrix},
\quad\text{where}\quad \Theta^{\prime}=\Lambda \,\Theta\,\Lambda^{-1}=\begin{pmatrix}
T^{\prime}_{11} & T^{\prime}_{12} \\
T^{\prime}_{21} & T^{\prime}_{22}
\end{pmatrix}\!,\ \Lambda \in G,
$$
and $v^\prime = \pm \left(-(\,\Lambda\,\Theta\,\Lambda^{-1} - T_{33}1\!\!1_2\,)\,u + \Lambda v\,\right)$.
In looking for simplified canonical forms for $T$ we shall first search for conditions
under which $\Theta^\prime$ can be brought to upper triangular form. 
It is not difficult to prove that this can be done
if and only if $T_{11} \neq T_{22}$ or $T_{21} \neq -T_{12}$.

\smallskip
Observe that $
\Theta=\left(
\begin{smallmatrix}
T_{11} & T_{12} \\
T_{21} & T_{22}
\end{smallmatrix}
\right)
$ 
satisfies $T_{11}=T_{22}$ and $T_{12}=-T_{21}$
if and only if there is an $\Lambda=
\left(
\begin{smallmatrix}
a& -b \\ b&\,\, a
\end{smallmatrix}
\right)\in G$ with $b\ne 0$, such that,
$\Lambda\,\Theta\,\Lambda^{-1}=\Theta$. 
Let $g =\left(\begin{smallmatrix}\pm 1\!\!1_2 & \,u\\ 0 & \,1\end{smallmatrix}\right)$
with $u=\left(\begin{smallmatrix}c\\ d\end{smallmatrix}\right)\in\C^2$. Then,
$$
g\,\,T\,g^{-1}=
\begin{pmatrix}
\Theta &\mp(\,\Theta - T_{33}1\!\!1_2\,)\,u + v\\
0 & T_{33}
\end{pmatrix}\!, \quad v=\begin{pmatrix}
T_{13} \\ T_{23}
\end{pmatrix}.
$$
If $\det(\Theta - T_{33}1\!\!1_2) \neq 0$ then the linear map $\Theta - T_{33}1\!\!1_2$ is invertible. So, given $v \in \C^2$, there exists $u \in \mathbb{C}^2$, such that $\mp(\Theta - T_{33}1\!\!1_2)u + v=0$.

\smallskip
We may now proceed to give a 
complete list of canonical
forms for the linear maps $T \in \operatorname{HL}(\mu_{2}^{\prime\prime})$.
They come out divided into two non-isomorphic families;
those for which $\Theta$ can be brought to upper triangular form
and those who cannot. 

\begin{Prop}\label{D2-2}
{\sl Let $\{e_1,e_2,e_3\}$ be the basis of $\g$
with respect to which the degenerate product $\mu^{\prime\prime}_2$ 
is given by:
$$
\mu^{\prime\prime}_2(e_1,e_2)=0,\quad \mu^{\prime\prime}_2(e_2,e_3)=e_1,\quad \mu^{\prime\prime}_2(e_3,e_1)=e_2.
$$
Each $T \in \operatorname{HL}(\mu^{\prime\prime}_2)$ as in \eqref{TenH(mu2)}
is equivalent to one and only one of the following canonical forms: \/}
\begin{enumerate}

\item {\sl If $\det(\Theta - T_{33} 1\!\!1_2)\ne 0$,
and either $T_{11}\ne T_{22}$ or $T_{12}\ne -T_{21}$, then
$$
T \simeq \begin{pmatrix}
 T^{\prime}_{11} & T^{\prime}_{12} & 0 \\
0 & T^{\prime}_{22} & 0 \\
0 & 0 & T^{\prime}_{33}
\end{pmatrix}\!,\quad \mbox{ with }T^{\prime}_{11} \ne T^{\prime}_{22} \mbox{ or }T^{\prime}_{12} \ne 0.
$$
\/}
\item {\sl If $\det(\Theta - T_{33}1\!\!1_2)=0$, and either $T_{11}\ne T_{22}$
or $T_{12}\ne -T_{21}$, then\/}
$$
T \simeq \begin{pmatrix}
T^{\prime}_{33} & T^{\prime}_{12} & T^{\prime}_{13} \\
0 & T^{\prime}_{22} & T^{\prime}_{23} \\
0 & 0 & T^{\prime}_{33}
\end{pmatrix}\!,\quad \mbox{ with }T^{\prime}_{11} \ne T^{\prime}_{22} \mbox{ or }T^{\prime}_{12} \ne 0.
$$
$$
\quad{or}\quad
T \simeq 
\begin{pmatrix}
T^{\prime}_{11} & T^{\prime}_{12} & T^{\prime}_{13} \\
0 & T^{\prime}_{33} & T^{\prime}_{23} \\
0 & 0 & T^{\prime}_{33}
\end{pmatrix}\!,\quad \mbox{ with }T^{\prime}_{11} \ne T^{\prime}_{22} \mbox{ or }T^{\prime}_{12} \ne 0.
$$
and these two are themselves equivalent if and only if $T^{\prime}_{12}=0$.

\item {\sl If $\det(\Theta - T_{33} 1\!\!1_2)\ne 0$, $T_{11} = T_{22}$
and $T_{12}= -T_{21}$, then\/}
$$
T \simeq \begin{pmatrix}
T^{\prime}_{11} & \!-T^{\prime}_{12} & 0 \\
T^{\prime}_{12} & \,\,T^{\prime}_{11} & 0 \\
0 & \,\,0 & T^{\prime}_{33}
\end{pmatrix}.
$$

\item {\sl If $\det(\Theta - T_{33} 1\!\!1_2)= 0$, $T_{11} = T_{22}$
and $T_{12}= -T_{21}$, then 
$$
T \simeq \begin{pmatrix}
T^{\prime}_{11} & \!-T^{\prime}_{12} & T^{\prime}_{13} \\
T^{\prime}_{12} & \,\,T^{\prime}_{11} & T^{\prime}_{23} \\
0 & \,\,0 & T^{\prime}_{11}+ iT^{\prime}_{12}
\end{pmatrix},
\quad\text{or}\quad
T \simeq \begin{pmatrix}
T^{\prime}_{11} & \!-T^{\prime}_{12} & T^{\prime}_{13} \\
T^{\prime}_{12} & \,\,T^{\prime}_{11} & T^{\prime}_{23} \\
0 & \,\,0 & T^{\prime}_{11}- iT^{\prime}_{12}
\end{pmatrix},
$$
and these two are equivalent.\/}
\end{enumerate}
\end{Prop}

\subsubsection{HL-algebras for the product $\mu^{\prime\prime}_3$\,}

Choose the basis of $\g$ so that $\mu^{\prime\prime}_3
=\left(
\begin{smallmatrix}
0 & \,1 & \,\,0 \\
1 & \,0 & \!-1 \\
0 & \,1 & \,\,0
\end{smallmatrix}\right)$.
It is not difficult to prove that 
the isotropy subgroup $G_{\mu^{\prime\prime}_3}$ consists only of the identity map. 
However, we may use a different basis to simplify a bit further
the matrix form of $\mu^{\prime\prime}_3$. Namely, there is
a basis $\{e^{\prime}_1,e^{\prime}_2,e^{\prime}_3\}$ such that
$$
\mu^{\prime\prime}_3(e^{\prime}_1,e^{\prime}_2)=-e^{\prime}_1,\quad \mu^{\prime\prime}_3(e^{\prime}_2,e^{\prime}_3)=e^{\prime}_2+e^{\prime}_3,\quad \mu^{\prime\prime}_3(e^{\prime}_3,e^{\prime}_1)=e^{\prime}_1.
$$
Let $e^{\prime\prime}_1=e^{\prime}_3$, $e_2^{\prime\prime}=e^{\prime}_1$ and $e_3^{\prime\prime}=e^{\prime}_2+e^{\prime}_3$. In this basis, $\mu^{\prime\prime}_3$ takes the form:
\begin{equation}\label{productoMU3}
\mu^{\prime\prime}_3 = 
\begin{pmatrix}
\,\,0 & \,0 & \,0 \\
\!-2 & \,0 & \,1 \\
\,\,0 & \,1 & \,0
\end{pmatrix}.
\end{equation}
A straightforward argument now proves the following:

\begin{Prop}\label{D2-3}
{\sl Let $\g$ be a complex $3$-dimensional vector space
and fix the basis so that the degenerate product
$\mu^{\prime\prime}_3$ takes the form \eqref{productoMU3}.
Then, $T \in \operatorname{HL}(\mu^{\prime\prime}_3)$ if and only if, its matrix has the form:
$$
\begin{pmatrix}
T_{11} & 0 & T_{13} \\
T_{21} & T_{22} & T_{23} \\
T_{31} & T_{32} & T_{33}
\end{pmatrix}\!\!,\quad 2\operatorname{Tr}(T)=T_{13}-4T_{31}.
$$\/}
\end{Prop}

\subsubsection{HL-algebras for the product $\mu^{\prime\prime}_4$.}

Choose the basis $\{e_1,e_2,e_3\}$ of $\g$ so that the 
degenerate product $\mu^{\prime\prime}_4$ takes the matrix form,
$\mu^{\prime\prime}_4=\left(\begin{smallmatrix}
\,\,1 & \,1 & \,0 \\
\!-1 & \,1 & \,0 \\
\,\,0 & \,0 & \,0
\end{smallmatrix}\right)$.
Let $e^{\prime}_1=e_3$, $e^{\prime}_2=ie_1+e_2$, and $e^{\prime}_3=-ie_1+e_2$.
In the basis $\{e^{\prime}_1,e^{\prime}_2,e^{\prime}_3\}$, we have,
\begin{equation}\label{MatrizDeMu4}
\mu^{\prime\prime}_4 =
\sqrt{2i}\begin{pmatrix}
0 & 0 & 0\\
0 & 0 & 1\\
0 & i & 0
\end{pmatrix}
\simeq \begin{pmatrix}
0 & 0 & 0 \\
0 & 0 & 1 \\
0 & i & 0
\end{pmatrix}.
\end{equation}
This product is a Lie algebra bracket.
Its isotropy subgroup is,
$$
G_{\mu^{\prime \prime}_4}=\left\{\,
\begin{pmatrix}
1 & 0 & 0\\
c & a & 0\\
d & 0 & b\\ 
\end{pmatrix}\,
\Bigg\vert\ \, ab \neq 0;\  c,d \in \mathbb{C}\,
\right\}.
$$
Also, in terms of the same basis $\{e^{\prime}_1,e^{\prime}_2,e^{\prime}_3\}$
we have,
\begin{equation}\label{HL(mu4)}
\operatorname{HL}(\mu^{\prime\prime}_4)
=\left\{\,
\begin{pmatrix}
T_{11} & 0 \\
v & \Theta
\end{pmatrix}
\,\Bigg\vert\ \, 
T_{11} \in \mathbb{C},\  v \in \mathbb{C}^2,\  \Theta \in \Mat_{2 \times 2}(\mathbb{C})\,
\right\}.
\end{equation}
It is easy to see that for any $g\in G_{\mu^{\prime \prime}_4}$
and $T\in \operatorname{HL}(\mu^{\prime\prime}_4)$,
$$
g\,\,T\, g^{-1}=\begin{pmatrix}
T_{11} & 0 \\
\Lambda\,(\,T_{11}1\!\!1_2-\Theta\,)\,\Lambda^{-1}u +\Lambda \,v & \Lambda \,\Theta \,\Lambda^{-1}
\end{pmatrix}\!.
$$
If $T_{11}$ is not an eigenvalue of $\Theta$, then there exists $u \in \mathbb{C}^2$, such that 
$(\Theta-T_{11}1\!\!1_2)\,u-v=0$. Therefore, we have the following:

\begin{Prop}\label{D2-4}
{\sl Let $\g$ be a complex $3$-dimensional vector space
and fix the basis so that the non-degenerate product
$\mu^{\prime\prime}_4$ takes the form \eqref{MatrizDeMu4}.
Any $T \in \operatorname{HL}(\mu^{\prime\prime}_4)$ as in \eqref{HL(mu4)}
is equivalent to a one and only one of the following canonical forms\/:}

\begin{enumerate}

\item {\sl If $\det(\Theta-T_{11}1\!\!1_2) \neq 0$, then\/}
$$
T \simeq \begin{pmatrix}
T^{\prime}_{11} & 0 \\
0 & \Theta^{\prime}
\end{pmatrix}\!,\quad \mbox{ with }\,T^{\prime}_{11} \in \C,\mbox{ and } \det(\Theta-T_{11}1\!\!1_2) \neq 0.
$$
\item {\sl If $\det(\Theta-T_{11}1\!\!1_2)=0$, then
$$
T \simeq \begin{pmatrix}
T^{\prime}_{11} & 0 \\
v^{\prime} & \Theta^{\prime}
\end{pmatrix}\!,\quad \mbox{ with }\,T^{\prime}_{11} \in \C,\mbox{ and } v^{\prime} 
\in \C^2,\quad \det(\Theta^{\prime}-T^{\prime}_{11}1\!\!1_2)=0.
$$\/}
\end{enumerate}
\end{Prop}

\subsection{HL-algebras for degenerate products of rank 1} 

\subsubsection{HL-algebras for the product $\mu^{\prime}_1$}

\begin{Prop}\label{D1-1}
{\sl Let $\{e_1,e_2,e_3 \}$ be the basis of $\g$ with respect to which the degenerate product $\mu^{\prime}_1$ is given by:
$$
\mu^{\prime}_1(e_1,e_2)=e_2,\quad \mu^{\prime}_1(e_2,e_3)=e_1,\quad \mu^{\prime}_1(e_3,e_1)=-e_3.
$$
Its isotropy subgroup is,
$$
G_{\mu^{\prime}_1}=\left\{\,
g=\begin{pmatrix}
1 & 0 \\
0 & \Lambda
\end{pmatrix}
\Bigg\vert\ \Lambda=\begin{pmatrix}
a&b\\c&d\end{pmatrix},
\ \det\Lambda =1\,
\right\}.
$$
Then, $T=(T_{ij}) \in \operatorname{HL}(\mu_{1}^{\prime})$, if and only if:
$$
T=\begin{pmatrix}
T_{11} & u^t \\
Ju & \Theta
\end{pmatrix},\ \  u=\begin{pmatrix}
T_{12} \\
T_{13}
\end{pmatrix},
\ \ J=\begin{pmatrix}
0 & 1 \\
-1 & 0
\end{pmatrix},\ \ 
\Theta=\begin{pmatrix}
T_{22} & T_{23} \\
T_{32} & -T_{22}
\end{pmatrix}.
$$
Then,  $T$ is equivalent to one and only one of the following canonical forms:\/}
$$
\alignedat 3
\text{\rm (1)}&\ \text {\sl If $\det \Theta=0$, then \/}\quad &
T \simeq \begin{pmatrix}
T^{\prime}_{11} & T^{\prime}_{12} & T^{\prime}_{13} \\
T^{\prime}_{13} & 0 & T^{\prime}_{23} \\
-T^{\prime}_{12} & 0 & 0
\end{pmatrix}.\quad &
\\
\text{\rm (2)}&\ \text {\sl If $\det \Theta \neq 0$, then \/}\quad &
T \simeq \begin{pmatrix}
T^{\prime}_{11} & T^{\prime}_{12} & T^{\prime}_{13} \\
T^{\prime}_{13} & T^{\prime}_{22} & 0 \\
-T^{\prime}_{12} & 0 & -T^{\prime}_{22}
\end{pmatrix},\quad & \text{with,}\ T^{\prime}_{22} \neq 0.
\endaligned
$$
\end{Prop}

\begin{proof}
Let $T=(T_{ij}) \in \operatorname{HL}(\mu^{\prime}_1)$ and let $g \in G_{\mu^{\prime}_1}$. 
Then,
$$
g\,\, T\,g^{-1}=
\begin{pmatrix}
T^{\prime}_{11} & u^t \Lambda^{-1}\\
\Lambda Ju & \Lambda \Theta \Lambda^{-1}
\end{pmatrix}
=
\begin{pmatrix}
T^{\prime}_{11} & {u^{\prime}}^t \\
Ju^{\prime} & \Theta^{\prime}
\end{pmatrix}\!
=
\begin{pmatrix}
T_{11}^{\prime} & T_{12}^{\prime} & T^{\prime}_{13} \\
T^{\prime}_{13} & T_{22}^{\prime} & T^{\prime}_{23} \\
-T^{\prime}_{12} & T^{\prime}_{32} & -T^{\prime}_{22}
\end{pmatrix}\!.
$$
Where, $T'_{11}= T_{11}$, $u'= (\Lambda^{-1})^t u$ and $\Theta^{\prime}= \Lambda \Theta \Lambda^{-1}$. 
We know there exists an invertible matrix $P \in \operatorname{GL}_2(\C)$ 
such that $P \Theta P^{-1}$ is upper triangular. 
Letting $\Lambda=(\det P)^{-1/2}\,P$, we make
$\det\Lambda =1$ and $\Lambda \Theta \Lambda^{-1}$ remains upper triangular. 
Thus, we may assume that $T^{\prime}_{32}=0$.
Suppose $\det \Theta^\prime \neq 0$. 
It is not difficult to see that there exists an automorphism $g
\in G_{\mu^{\prime}_1}$, that brings $\Lambda \Theta^\prime \Lambda^{-1}$ into diagonal form. The only other alternative is to have $\det \Theta^\prime = 0$, from which the statement follows.
\end{proof}

\subsubsection{HL-algebras for the product $\mu^{\prime}_2$\,}
The product $\mu^{\prime}_2$ is the Lie bracket of
the $3$-dimensional Heisenberg Lie algebra. The HL-algebras based
on $\mu^{\prime}_2$ have been thoroughly studied in \cite{Alejandra}, 
using $T\in\operatorname{HL}(\mu^{\prime}_2)\cap\operatorname{Aut}(\g;\mu^{\prime}_2)$.
It turns out, however, that any linear map $T:\g\to\g$ satisfies the HL-Jacobi identity for 
$\mu^{\prime}_2$; {\it ie\/,} $\operatorname{HL}(\mu^{\prime}_2)=\End_{\mathbb{C}}(\g)$.
Thus, in contrast to the approach in  \cite{Alejandra},
we consider here not only Lie algebra automorphisms, but the most general linear maps
$T\in \End_{\mathbb{C}}(\g)$.

\begin{Prop}\label{D1-2}
{\sl Let $\mu^{\prime}_2$ be the degenerate product defined on a basis $\{e_1,e_2,e_3\}$ by,
$$
\mu^{\prime}_2(e_1,e_2)=\mu^{\prime}_2(e_3,e_1)=0,\quad \mu^{\prime}_2(e_2,e_3)=1.
$$
Its isotropy subgroup is,
$$
G_{\mu^{\prime}_2}=\left\{\,
\begin{pmatrix}
\alpha & u^t \\
0 & \Lambda
\end{pmatrix}
\Bigg\vert\ \, u \in \mathbb{C}^2,\ \ \det \Lambda=\alpha\ne 0\,
\right\}.
$$
Then, any $T =(T_{ij}) \in\End_{\Bbb C}(\g)$ belongs to $\operatorname{HL}(\mu^{\prime}_2)$
and such a $T$ is equivalent to one and only one of the following canonical forms:\/}

\begin{enumerate}

\item {\sl If $(T_{21},T_{31}) \neq (0,0)$, then
$$
T \simeq \begin{pmatrix}
T^{\prime}_{11} & T^{\prime}_{12} & T^{\prime}_{13} \\
T^{\prime}_{21} & T^{\prime}_{22} & 0 \\
T^{\prime}_{31} & 0 & 0
\end{pmatrix}\!, \quad \mbox{ with }(T^{\prime}_{21},T^{\prime}_{31}) \neq (0,0).
$$
\/}

\item {\sl If $(T_{12},T_{13})=(0,0)$ and $(T_{22}-T_{11})(T_{33}-T_{11}) \neq 0$, then 
$$
T \simeq \begin{pmatrix}
T^{\prime}_{11} & 0 & 0 \\
0 & T^{\prime}_{22} & T^{\prime}_{23} \\
0 & 0 & T^{\prime}_{33}
\end{pmatrix}\!\mbox{ with } (T^{\prime}_{22}-T^{\prime}_{11})(T^{\prime}_{33}-T^{\prime}_{11})\neq 0.
$$
\/}

\item {\sl If $(T_{12},T_{13})=(0,0)$ and $(T_{22}-T_{11})(T_{33}-T_{11})=0$, then
$$
T \simeq \begin{pmatrix}
T^{\prime}_{11} & T^{\prime}_{12} & T^{\prime}_{13} \\
0 & T^{\prime}_{22} & T^{\prime}_{23} \\
0 & 0 & T^{\prime}_{33}
\end{pmatrix}\!,\quad \mbox{ with } (T^{\prime}_{22}-T^{\prime}_{11})(T^{\prime}_{33}-T^{\prime}_{11})=0.
$$
\/}

\end{enumerate}
\end{Prop}

\begin{proof}
Write the matrix of $T$ in the form,
$$
T=\begin{pmatrix}
T_{11} & w^t \\
v & \Theta
\end{pmatrix}\!,\quad T_{11} \in \mathbb{C}, \quad v=\begin{pmatrix}
T_{21}\\
T_{31}
\end{pmatrix}\!,\,w=\begin{pmatrix}
T_{12}\\T_{13}
\end{pmatrix}\in \mathbb{C}^2,
$$ 
$$
\Theta=\begin{pmatrix}
T_{22}&T_{23}\\T_{32}&T_{33}
\end{pmatrix} \in \Mat_{2\times 2}(\mathbb{C}).
$$
For any $g \in G_{\mu^{\prime}_2}$, the matrix $g\,\,T\,g^{-1}$ is equal to,
$$
\begin{pmatrix}
T_{11}+\alpha^{-1}u^tv & -(T_{11}+\alpha^{-1}u^tv)u^t\Lambda^{-1}+\alpha w^t\Lambda^{-1}+u^t\Theta\Lambda^{-1} \\
\alpha^{-1} \Lambda v & -\alpha^{-1}\Lambda vu^t\Lambda^{-1}+\Lambda \Theta \Lambda^{-1}
\end{pmatrix}\!\!.
$$
We denote by $\adj(\Theta)$ the 
unique matrix that satisfies
$\Theta\,\adj(\Theta)=\adj(\Theta)\,\Theta
=(\det \Theta)\,1\!\!1_2$.
Let, $\Theta^\prime=\Lambda\,\left(\,\Theta - \alpha^{-1}v\,u^t\,\right)\,\Lambda^{-1}$.
If $v \neq 0$, 
$u$ can be chosen so that
$T$ is equivalent to  $\left(\begin{smallmatrix}T_{11}^{\prime} & {w^{\prime}}^t \\
v^{\prime} & \Theta^{\prime}  \end{smallmatrix}\right)$, with 
$\det \Theta^{\prime}=0$.
Thus, we may always assume that $\Theta$ is singular whenever $v\ne 0$.
In particular, the characteristic polynomial of $\Theta$ is $x(x-\Tr(\Theta))$. 
We claim that there exists $g \in G_{\mu^{\prime}_2}$, 
such that $g\,T\,g^{-1}=\left(\begin{smallmatrix} T_{11}^{\prime} & {w^{\prime}}^{t} \\v^{\prime} &  \Diag\{T^{\prime}_{22},0\} \end{smallmatrix} \right)$. 
We proceed by considering two cases: either $\Tr(\Theta) \neq 0$, or $\Tr(\Theta)=0$.

If $\Tr(\Theta) \neq 0$, then $\Theta$ is diagonalizable.
If $\Tr(\Theta)=0$ we may congujate $T$ by $g=\left(
\begin{smallmatrix}
1 & u^t \\
0 & 1\!\!1_2
\end{smallmatrix}\right)$, 
to obtain
$T^\prime=g\,T\,g^{-1}=\left(\begin{smallmatrix} T_{11}^{\prime} & {w^{\prime}}^{t} \\v^{\prime} &  \Theta^{\prime} \end{smallmatrix} \right)$, where:
$$
\det \Theta^{\prime}=-\langle u,\adj(\Theta)(v) \rangle,\quad \mbox{ and } \quad \Tr(\Theta^{\prime})=-\langle u,v \rangle,
$$
{\it for any\/} $u \in \mathbb{C}^2$.
Since $\det \Theta=0$, by considering $\Theta^t(u)$ instead of $u$, we see that, $\det \Theta^{\prime}=0$ and $\Tr(\Theta^{\prime})=-\langle u,\Theta(v) \rangle$, where $u$ is \emph{any} element in $\C^2-\{0\}$. 
Thus, if $\Theta(v) \neq 0$, we may choose $u \in \mathbb{C}^2$, such that $\langle u,\Theta(v) \rangle \neq 0$; implying, $\Tr(\Theta^{\prime}) \neq 0$.
Now, suppose $\Theta(v)=0$. If $\adj(\Theta)(v)=0$, then any $u \in \C^2$ with $\langle u,v \rangle \neq 0$, satisfies $\det \Theta^{\prime}=\langle u,\adj(\Theta)(v) \rangle=0$ and $\Tr(\Theta^{\prime})=\langle u,v \rangle \neq 0$, bringing us back again to the same canonical form.

\smallskip
On the other hand, assume $v^{\prime}=\adj(\Theta)(v) \neq 0$. Since $\det \Theta=0$, it follows that
$\Theta(v')=0$, which in turn implies that $v$ and $v^{\prime}$ both belong to $\Ker(\Theta)$. 
Since $\Theta \neq 0$, then $v$ and $v^{\prime}$ are linearly dependent, therefore there exists $\zeta \in \mathbb{C}-\{0\}$ such that $v^{\prime}=\zeta v$. This means that $\adj(\Theta)(v)=\zeta v$, and since $\Theta$ is singular, this in turn leads us to $\Tr(\Theta) \neq 0$, in contradiction with our hypotheses.
Therefore, $\adj(\Theta)(v)=0$. In summary, 
we have proved that if $T=\left(
\begin{smallmatrix} T_{11} & w^t\\ v & \Theta\end{smallmatrix}
\right)$, with $v\ne 0$, then there is
a $g \in G_{\mu^{\prime}_2}$, such that,
$$
T^\prime=g\,\,T\,g^{-1}=\begin{pmatrix}
T^{\prime}_{11} & T^{\prime}_{12} & T^{\prime}_{13} \\
T^{\prime}_{21} & T^{\prime}_{22} & 0 \\
T^{\prime}_{31} & 0 & 0
\end{pmatrix}\!.
$$
Now consider the case $v=0$, so that the linear map $T \in \operatorname{HL}(\mu^{\prime}_2)$ 
has the form $T=\left(\begin{smallmatrix} T_{11} & w^t \\
0 & \Theta \end{smallmatrix}\right)$.
We may assume with no loss of generality that $\Theta$ is upper triangular. 
Conjugation by 
$g=\left(\begin{smallmatrix}
1 & u^t \\
0 & 1\!\!1_2
\end{smallmatrix}\right)\in G_{\mu^\prime_2}$, yields,
$$
T^\prime=g\,\,T\,g^{-1}=
\begin{pmatrix}
T_{11} & u^t(\Theta-T_{11}1\!\!1_2)+w^t \\
0 & \Theta
\end{pmatrix}\!.
$$
Further simplification of $T^\prime$ to obtain the canonical forms
now depends on whether or not $T_{11}$ is an eigenvalue of $\Theta$
and the argument goes as in previous cases. 
\end{proof}

\subsubsection{HL-algebras for the product $\mu^{\prime}_3$}

\begin{Prop}\label{D1-3}
{\sl Let $\{e_1,e_2,e_3 \}$ be a basis of $\g$ for which the
degenerate product $\mu^{\prime}_3$ is given by,
$$
\mu^{\prime}_3(e_1,e_2)=e_1,\quad \mu^{\prime}_3(e_2,e_3)=e_1-e_3,\quad \mu^{\prime}_3(e_3,e_1)=0,
$$ 
and its isotropy subgroup is,
$$
G_{\mu^{\prime}_3}=\left\{\,
\begin{pmatrix}
a & b & c\\
0 & 1 & 0\\
0 & d & a
\end{pmatrix}
\Bigg\vert\ \,  a \neq 0,\,b,c,d \in \C
\right\}.
$$
Then, each $T \in \operatorname{HL}(\mu^{\prime}_3)$ is equivalent 
to one and only one of the following canonical forms: \/}
$$
\alignedat 3
\text{\rm (1)}&\ \text{\sl If $T_{31}=0$, then\/}\quad &
T \simeq \begin{pmatrix}
T^{\prime}_{11} & T^{\prime}_{12} & T^{\prime}_{13} \\
0 & T^{\prime}_{22} & 0 \\
0 & T^{\prime}_{32} & T^{\prime}_{33}
\end{pmatrix}.\quad &
\\
\text{\rm (2)}&\ \text{\sl If $T_{31} \neq 0$, then\/}\quad &
T \simeq \begin{pmatrix}
T^{\prime}_{11} & T^{\prime}_{12} & T^{\prime}_{13} \\
0 & T^{\prime}_{22} & 0 \\
T^{\prime}_{31} & 0 & 0
\end{pmatrix}\!,\quad & \text{with,}\ T^{\prime}_{31} \neq 0.
\endaligned
$$
\end{Prop}

\begin{proof}
Observe that $T \in \operatorname{HL}(\mu^{\prime}_3)$ if and only if its matrix
$T=(T_{ij})$ in the given basis, satisfies $T_{21}=0 = T_{23}$.
Then, for $g \in G_{\mu^{\prime}_3}$, we have, $T^{\prime}=g\,\,T\,g^{-1}$,
where,
$$
\aligned
T^{\prime}_{31}=& T_{31},\\
T^{\prime}_{32}=& (a^{-1}cd-b)T_{31}+aT_{32}+d(T_{22}-T_{33}),\\
T^{\prime}_{33}=&-a^{-1}c \, T_{31}+T_{33}.
\endaligned
$$
If $T_{31} \neq 0$, we may choose $a,b,c,d \in \C$, 
so that $T_{32}^{\prime}=T_{33}^{\prime}=0$. 
Thus, there are two non-isomorphic canonical forms, depending 
on whether $T_{31}$ is equal to zero or not, as claimed.

\end{proof}

\subsection{$\operatorname{HL}$-algebras for the product $\mu_0$}

\begin{Prop}
{\sl Let $\{e_1,e_2,e_3\}$ be a basis of $\g$ for which the degenerate product 
$\mu_0$ is given by,
$$
\mu_0(e_1,e_2)=e_2,\quad \mu_0(e_2,e_3)=0,\quad \mu_0(e_3,e_1)=-e_3,
$$
and its isotropy subgroup is,
$$
G_{\mu_0}=\left\{\,
\begin{pmatrix}
1 & 0 \\
u & A
\end{pmatrix}
\Bigg\vert\ \,u=\begin{pmatrix}g_{21} \\ g_{31} \end{pmatrix} \ \text{and}\ \ 
\Lambda=\begin{pmatrix}
g_{22} & g_{23} \\g_{32} & g_{33}
\end{pmatrix}\!\in \GL_2(\C)
\right\}.
$$
The product $\mu_0$ corresponds to a Lie bracket. 
In fact $(\g,\mu_0)$ is a 3-dimensional non-nilpotent solvable Lie algebra.
Moreover, $T \in \operatorname{HL}(\mu_0)$ if and only if, 
its matrix has the form $\left(\begin{smallmatrix} T_{11} & 0 \\ v & B \end{smallmatrix}\right)$, 
where $v=\left(\begin{smallmatrix}T_{21} \\ T_{31} \end{smallmatrix}\right)$, and 
$B \in \operatorname{Mat}_{2 \times 2}(\C)$. 
Then, $T$ is equivalent to one and only one of the following canonical forms\/:}
\begin{enumerate}

\item {\sl If $\det(T_{11}\,1\!\!1_2-B) \neq 0$, either\/,}
$$
T \simeq \begin{pmatrix}
T^{\prime}_{11} & 0 & 0 \\
0 & T^{\prime}_{22} & 0 \\
0 & 0 & T^{\prime}_{33}
\end{pmatrix}\!,\quad \mbox{ with }
\quad (T^{\prime}_{11}-T^{\prime}_{22})(T^{\prime}_{11}-T^{\prime}_{33})\neq 0,
$$
$$
\text{or,}\qquad 
T \simeq \begin{pmatrix}
T^{\prime}_{11} & 0 & 0 \\
0 & T^{\prime}_{22} & 0 \\
0 & 1 & T^{\prime}_{22}
\end{pmatrix}\!,\quad \mbox{ with }\quad T^\prime_{11}\ne T^\prime_{22}.
$$

\item {\sl If $\det(T_{11}\,1\!\!1_2-B)=0$, either\/,}
$$
T \simeq \begin{pmatrix}
T^{\prime}_{11} & 0 & 0 \\
T^{\prime}_{21} & T^{\prime}_{22} & 0 \\
T^{\prime}_{31} & 0 & T^{\prime}_{33}
\end{pmatrix}\!,\quad \mbox{ with }\quad (T^{\prime}_{11}-T^{\prime}_{22})(T^{\prime}_{11}-T^{\prime}_{33})=0.
$$
$$
\text{or,}\qquad 
T \simeq \begin{pmatrix}
T^{\prime}_{11} & 0 & 0 \\
T_{21}^{\prime} & T^{\prime}_{11} & 0 \\
T_{31}^{\prime} & 1 & T^{\prime}_{11}
\end{pmatrix}.
$$
\end{enumerate}
\end{Prop}

\begin{proof}
It is a straightforward matter to verify that $(\g,\mu_0)$
is a 3-dimensional non-nilpotent solvable Lie algebra.
A direct computation also shows that $T \in \operatorname{HL}(\mu_0)$
if and only if it has the matrix form given in the statement.
For such a $T$ and $g \in G_{\mu_0}$, we obtain,
\begin{equation}\label{mu cero}
g\,\,T\,g^{-1}=\begin{pmatrix}
T_{11} & 0 \\
\Lambda((T_{11}\,1\!\!1_2-B)\Lambda^{-1}u+w) & \Lambda B\Lambda^{-1}
\end{pmatrix}\!.
\end{equation}
We can choose $\Lambda \in \GL_2(\C)$ to bring $ \Lambda B\Lambda ^{-1}$
to its Jordan form. 
On the other hand, 
if $\det(T_{11}\,1\!\!1_2-B) \neq 0$, 
we can find $u \in \C^{2}$, such that $(T_{11}\,1\!\!1_2-B)(u)=-w$. 
Otherwise, one of the entries in the diagonal of the Jordan form of 
$B$ must be equal to $T_{11}$. 
\end{proof}

\section*{Acknowledgements}
The authors would also like to thank the referee for his/her comments,
criticism and valuable suggestions, as they gave them the opportunity to produce
a better exposition of their results by susbstantially improving the original presentation.
The authors also acknowledge the support received
through CONACyT Grant $\#$ A1-S-45886. 
The author RGD would also like to thank the support 
provided by the post-doctoral fellowships
FOMIX-YUC 221183
and FORDECYT 265667
that allowed him to
stay at CIMAT - Unidad M\'erida.
Finally, GS acknowledges the
support provided by PROMEP grant UASLP-CA-228
and ASV acknowledges the support given by MB1411. 
\bibliographystyle{amsalpha}

\begin{thebibliography}{31}

\frenchspacing


\bibitem{Alejandra} 
\'Alvarez, M. A., and Cartes, F.,
\emph{Cohomology and deformations for the Heisenberg Hom-Lie algebras},
Linear and Multilinear Algebra (2019) 2209-2229.
\hfill\break
https://doi.org/10.1080/03081087.2018.1487379


\bibitem{Garcia}
Garc\'{i}a-Delgado R., 
\emph{Generalized derivations and Hom-Lie algebra structures on $\frak{sl}_2$.} 
Preprint at arXiv:1903.03672v5 [math.RA] \hfill\break
https://arxiv.org/pdf/1903.03672v5.pdf


\bibitem{H-M-MM}
Hern\'andez Encinas, L.,  Mart\'in del Rey, A., and Mu\~noz Masqu\'e, J.
\emph{Non-degenerate bilinear alternanting maps $f:V \times V \to V$, $\dim V = 3$,
over an algebraically closed field.}
Linear Alg. and its App. {\bf 387} (2004) 69-82.

\bibitem{Hum} Humphreys, J. E., \emph{Introduction to Lie Algebras and Representation
Theory}, Springer-Verlag, United States of America, (1972).

\bibitem{Remm}
Remm E., \emph{3-Dimensional Skew-symmetric Algebras and the Variety of Hom-Lie Algebras}, 
Alg. Coll. {\bf 25} (2018) 547-566. 
\hfill\break 
https://doi.org/10.1142/S100538671800038X

\bibitem{Xiao}
Xiao-chao L., Yong-feng L., \emph{Classification of 3-dimensional multiplicative Hom-Lie algebras}, J. of Xinyang Normal Univ. {\bf 25} (2012) 427-455.


\bibitem{Xie} Xie, W\., Jin, Q\., Liu, W., 
\emph{Hom-structures on semi-simple Lie algebras}, 
Open Mathematics {\bf 13} (1) $(2015)$ 617-630. \hfill\break
https://doi.org/10.1515/math-2015-0059


\end{thebibliography}

\end{document}